\documentclass{amsart}
\usepackage{amsmath}
\usepackage{amssymb}
\usepackage{mathrsfs}
\usepackage{stmaryrd}
\usepackage[all]{xy}
\usepackage{amsfonts}
\usepackage{xspace}

\makeatletter \renewcommand{\everyentry@}{\vphantom{A_{[]}{[]}}}
\makeatother

\newtheorem{theorem}{Theorem}[subsection]

\newtheorem{lemma}[theorem]{Lemma}
\newtheorem{prop}[theorem]{Proposition}
\newtheorem{co}[theorem]{Corollary}
\newtheorem{conj}[theorem]{Conjecture}
\theoremstyle{definition}

\newtheorem{question}[theorem]{Question}
\theoremstyle{remark}
\newtheorem{remark}[theorem]{Remark}

\numberwithin{equation}{subsection}

\def \a {\mathfrak a}

 \def \E{\mathcal E}

\def \Z {\mathbb Z}
\def \inj {\hookrightarrow }
\def \to {\rightarrow}
\def \spec \text{spec}
 \def \M {\mathfrak M}

\def \e { \underline \epsilon}

\def \p {\underline \pi}

\def \hL {{\hat {\mathfrak L}}}

\def \L {\mathfrak L}
\def  \cris {\textnormal{cris}}
\DeclareMathOperator{\gal}{Gal}

\def \Q {\mathbb Q}
\def \t {\textnormal}
\def \Z {\mathbb Z}

\def \ito {\overset  \sim  \to}
\def \O {\mathcal O}
\def \gs {\mathfrak S}

\def \ur {\t{ur}}
\def \D {\mathcal D}

\def \me {\mathfrak e}

\def \f {\mathfrak f}
\def \F {\mathbb F}

\def \Fil {\t{Fil}}

\def \CM {\mathcal M}

\def \acris {{A_{\t{cris}}}}

\def \R {\mathcal R}

\def \trep {{\t{Rep}_\t{tor}^{\t{ss}, r  }(G) }}

\def \< {\left <}
\def \> {\right >}

\def \hR {{\widehat \R} }
\def \hM {{\hat \M}}

\def \smin {s_{\t{min}}}
\def \card {\t{Card}\:}

\def\O{\mathcal O}
\def\calM{\mathcal M}

\begin{document}

\title{Some Bounds for ramification of $p^n$-torsion semi-stable representations}

\author{Xavier Caruso}
\address{IRMAR, Universit\'e Rennes 1, Campus de Beaulieu, 35042 Rennes Cedes, France.}
\email{xavier.caruso@normalesup.org}

\author{Tong Liu}
\address{Department of Mathematics, University of Pennsylvania, Philadelphia,19104, USA.}
\email{tongliu@math.upenn.edu, }

\subjclass{Primary  14F30,14L05}

\date{June 2008}



\maketitle



\begin{abstract}
Let $p$ be an odd prime, $K$ a finite extension of $\Q_p$, $G_K = \gal
(\bar K/K)$ its absolute Galois group and $e = e(K/\Q_p)$ its absolute
ramification index. Suppose that $T$ is a $p^n$-torsion representation
of $G_K$ that is isomorphic to a quotient of $G_K$-stable
$\Z_p$-lattices in a semi-stable representation with Hodge-Tate weights
$\{0, \ldots, r\}$. We prove that there exists a constant $\mu$
depending only on $n$, $e$ and $r$ such that the upper numbering
ramification group $G_K^{(\mu)}$ acts on $T$ trivially.
\end{abstract}

\tableofcontents
\section{Introduction}

\renewcommand\thetheorem{\thesection.\arabic{theorem}}

\newcommand{\tobecompleted}{(...)\xspace}
\newcommand{\GL}{\text{\rm GL}}

Let $p>2$ be a prime number and $k$ a perfect field of characteristic
$p$.  We denote by $W = W(k)$ the ring of Witt vectors with coefficients
in $k$. Fix $K$ a totally ramified
extension of $W[1/p]$ of degree $e$ and $\bar K$ an algebraic closure of
$K$. Fix $\pi \in \O_K$ a  uniformizer and $(\pi_s)_{s \geq 0}$ a
compatible system of $p^s$-th root of $\pi$. Set $G = \gal(\bar K/ K)$
and for all non negative integer $s$, put $K_s = K(\pi_s)$ and $G_s =
\gal(\bar K/K_s)$. Denote by $G^{(\mu)}$ and $G_s^{(\mu)}$ ($\mu \in
\mathbb R$) the upper ramification filtration of $G$ and $G_s$, as
defined in \S1.1 of \cite{intro:fontaine}. Note that conventions of
\emph{loc. cit.} differ by some shift with definition of \cite{serre},
Chap. IV. Finally, let $v_K$ be the discrete valuation on $K$ normalized
by $v_K(\pi) = 1$. It extends uniquely to a (not discrete) valuation on
$\bar K$, that we denote again $v_K$.

Consider $r$ a positive integer and $V$ a semi-stable representation of
$G$ with Hodge-Tate weights in $\{0, 1, \ldots, r\}$. Let $T$ be the
quotient of two $G$-stable $\Z_p$-lattices in $V$. It is a
representation of $G$, which is killed by $p^n$ for some integer $n$.
Denote by $\rho : G \to \t{Aut}_{\Z_p} (T)$ the associated group homomorphism and
by $L$ (resp. $L_s$) the finite extension of $K$ (resp. $K_s$) defined by
$\ker \rho$ (resp. $\ker \rho_{|G_s}$). We will prove:

\begin{theorem}
\label{intro:ramgs}
Keeping previous notations, for any integer $s > n + \log_p(\frac{nr}
{p-1})$ and for all real number $\mu > \frac{ern p^n}{p-1}$,
$G_s^{(\mu)}$ acts trivially on $T$.
\end{theorem}

\begin{remark}
Condition on $s$ implies $\frac{ern p^n} {p-1} < e p^s$. Hence one
may always choose $\mu = e p^s$.
\end{remark}

We also obtain a bound for the ramification of $L/K$:

\begin{theorem}
\label{intro:main}
Write $\frac{nr}{p-1} = p^\alpha \beta$ with $\alpha \in \mathbb N$ and
$\frac 1 p < \beta \leq 1$. Then:
\begin{enumerate}
\item if $\mu > 1 + e (n + \alpha + \max(\beta,\frac 1{p-1}))$, then
$G^{(\mu)}$ acts trivially on $T$;
\item $v_K(\D_{L/K}) < 1 + e (n + \alpha + \beta) - \frac
1{p^{n + \alpha}}$
\end{enumerate}
where $\D_{L/K}$ is the different of $L/K$.
\end{theorem}

Before this work, some partial results were already known in this
direction. First, in \cite{intro:fontaine} and \cite{intro:fontaine2},
Fontaine uses Fontaine-Laffaille theory (developped in \cite{intro:FL})
to get some bounds when $e = 1$, $n = 1$, $r < p-1$ and $V$ is
crystalline. In \cite{intro:abrashkin}, Abrashkin follows Fontaine's
general ideas to extend the result to arbitrary $n$ (other restrictions
remain the same). Later, with the extension by Breuil of
Fontaine-Laffaille theory to semi-stable case (see \cite{intro:breuil}),
it has been possible to achieve some cases where $V$ is not crystalline.
Precisely in \cite{intro:breuil2}\footnote{See Proposition 9.2.2.2 of
\cite{intro:BM} for the statement}, Breuil obtains bounds for
semi-stable representations that satisfies Griffith transversality when
$n = 1$ and $er < p-1$. Very recently in \cite{intro:hattori} and
\cite{intro:hattori2}, Hattori proves a bound for all semi-stable
representations with $r < p-1$ ($e$ and $n$ are arbitrary here).
Unfortunately, bounds found by those authors have the same shape than
ours but are in general slightly better (at least for $n>1$) and, in
fact, we conjecture
that Theorem \ref{intro:main} could be improved as follows:

\begin{conj}
\label{intro:conj}
Writing $\frac r {p-1} = p^{\alpha'} \beta'$ with $\alpha' \in \mathbb
N$ and $\frac 1 p < \beta' \leq 1$, we have:
\begin{enumerate}
\item if $\mu > 1 + e (n + \alpha' + \max(\beta',\frac 1 {p-1}))$, then
$G^{(\mu)}$ acts trivially on $T$;
\item $v_K(\D_{L/K}) < 1 + e (n + \alpha' + \beta')$.
\end{enumerate}
\end{conj}

\noindent
We also wonder if, in bounds given by the previous Conjecture, $\beta'$
can be replaced by $\frac{\beta'} p$ (remark then that $\max( \frac
{\beta'} p, \frac 1 {p-1})$ is just $\frac 1 {p-1}$), since it is
actually the case in Proposition 9.2.2.2 of \cite{intro:BM}.

\medskip

Note finally that the dependance in $r$ of bounds of Theorem
\ref{intro:main} is logarithmic, which is in fact quite surprising
since, until now all bounds seem to depend linearly on $r$. (Of course,
it does not mean anything since these bounds are valid under the
assumption for $r < p-1$, and certainly not for $r$ going to infinity.)
We finally wonder if better bounds exist when $V$ is crystalline. It is
actually the case when $e = 1$ and $r < p-1$ by results of Fontaine and
Abrashkin, but it is not clear to us how to extend this to a more
general setting.

\medskip

Let us now explain the general plan of our proof (and in the same time
of the article). For this we introduce first further notations: let
$K_\infty = \bigcup \limits_{s=1}^\infty K_s$ and $G_\infty = \gal(\bar
K/K_\infty)$. By some works of Fontaine, Breuil and Kisin, we know that
the restriction of $T$ to $G_\infty$ is described by some data of
(semi-)linear algebra that we will call in the sequel \emph{Kisin
modules}\footnote{In fact, these modules were first introduced by Breuil
in \cite{intro:breuil3} and \cite{intro:breuil4}. However, we think that
the terminology is not so bad since ``Breuil modules'' is already used
for other things and ``Kisin modules'' were actually intensively studied
by Kisin in \cite{intro:kisin1} and \cite{intro:kisin2}.}. Let's call it
$\M$. In the two following sections, we will show that the data of $\M$
is enough to recover the whole action of $G_s$ on $T$ for $s > \smin :=
n -1 + \log_p (n r)$.

More precisely, we first prove in section 2 (Theorem \ref{theo:prolong})
that any Kisin module killed by $p^n$ determines a canonical
representation of $G_s$ with $s > \smin$ (and not only $G_\infty$).
Note that this first step does not use any assumption of semi-stability:
our result is valid for all representations (killed by $p^n$) coming
from a Kisin module; no matter if it can be realized as a quotient of
two lattices in a semi-stable representation. Then, in section 3, we
show that the $G_s$-representation attached to $\M$ coincide with
$T|_{G_s}$. At this level, let us mention an interesting corollary of
the theory developed in these two sections:

\begin{theorem}[Corollary \ref{cor:pleinfid}]
\label{intro:pleinfid}
Let $V$ and $V'$ be two semi-stable representations of $G$. Let $T$
(resp. $T'$) a quotient of two $G$-lattices in $V$ (resp. $V'$) which is
killed by $p^n$. Then any morphism $G_\infty$-equivariant $f : T \to T'$
is $G_s$-equivariant for all integer $s > n - 1 + \log_p(nr)$.
\end{theorem}

Then, we conclude the proof of Theorem \ref{intro:ramgs} using usual
techniques developed by Fontaine in \cite{intro:fontaine}. Using some
kind of transitivity formulas, we then deduce Theorem \ref{intro:main}.
Finally, in the section \ref{sec:lift}, we begin a discussion about the
possibility, given a torsion representation of $G_K$, to write it as a
quotient of two lattices in a $\Q_p$-representation satisfying some
properties (like being crystalline, semi-stable, with prescribed
Hodge-Tate weights).

\subsubsection*{Conventions}

For any $\Z$-module $M$, we always use $M_n$ to denote $M/ p^n M$. If
$A$ be a ring, then $\t{M}_d(A)$ will denote the ring of $d \times
d$-matrices with coefficients in $A$. We reserve $\varphi$ to represent
various Frobenius structures (except that $\sigma$ stands for usual
Frobenius on $W(k)$) and $\varphi_M$ will denote the Frobenius on $M$.
But we always drop the subscript if no confusion arises.

Finally, if $A$ is a ring equipped with a valuation $v_A$ we will
often set:
$$\a_A^{\geqslant v} = \{ x \in A \, / \, v_A(x) \geq v\}
\quad \text{and} \quad
\a_A^{>v} = \{ x \in A \, / \, v_A(x) > v\}.$$

\section{$G_s$-representation attached to a torsion Kisin module}

\renewcommand\thetheorem{\thesubsection.\arabic{theorem}}

\newcommand{\Km}{\t{Mod}^{\varphi,r}_\gs}
\newcommand{\Kmn}[1][n]{\t{Mod}^{\varphi,r}_{\gs_{#1}}}
\newcommand{\Kminf}{\t{Mod}^{\varphi,r}_{\gs_\infty}}
\newcommand{\Kfmn}{\t{Free}^{\varphi,r}_{\gs_n}}
\newcommand{\Frac}{\t{Frac}\:}
\renewcommand{\hom}{\t{Hom}}

In this section, we prove that $G_\infty$-representation $T_{\gs_n}(\M)$
attached to Kisin modules $\M$ killed by $p^n$ can be naturally extended
to a $G_s$-representation for all $s > n -1 + \log_p(nr)$ (and sometimes
better).

\subsection{Definitions and basic properties of Kisin modules}

Recall the following notations: $k$ is a perfect field, $W = W(k)$, $K$
is a totally ramified extension of $W[1/p]$ of degree $e$, $\pi$ is a fixed
uniformizer of $K$ and $E(u)$ is the minimal polynomial of $\pi$. Recall
also that we have fixed a positive integer $r$. Define $E(u)$ to be the
minimal polynomial of $\pi$ over $W[1/p]$.

The base ring for Kisin modules is $\gs = W[\![u]\!]$. It is endowed with a
Frobenius map $\varphi: \gs \to \gs$ defined by:
$$\varphi\Big(\sum_{i \geq 0} a_i u^i\Big) = \sum_{i \geq 0}
\sigma(a_i) u^{pi}$$
where $\sigma$ stands for usual Frobenius on $W$. By definition, a
\emph{free Kisin module} (of height $\leq r$) is a $\gs$-module $\M$
free of finite rank equipped with a $\varphi$-semi-linear endomorphism
$\varphi_\M : \M \to \M$ such that the following condition holds:
\begin{equation}
\label{eq:condition}
\text{the }\gs\text{-submodule of }\M\text{ generated by }\varphi_\M(\M)
\text{ contains }E(u)^r \M.
\end{equation}
We denote by $\Km$ their category. Of course, a morphism is $\Km$ is
just a $\gs$-linear map that commutes with Frobenius actions. In the
sequel, if there is no risk of confusion, we will often write $\varphi$
instead of $\varphi_\M$.

There is also a notion of \emph{torsion Kisin modules} of height $\leq r$.
They are modules $\M$ over $\gs$ equipped with a $\varphi$-semi-linear
map $\varphi: \M \to \M$ such that:
\begin{itemize}
\item $\M$ is killed by a power of $p$;
\item $\M$ has no $u$-torsion;
\item condition (\ref{eq:condition}) holds.
\end{itemize}

Let us call $\Kminf$ (resp. $\Kmn$, resp. $\Kfmn$) the category of all
torsion Kisin modules (resp. of torsion Kisin modules killed by $p^n$,
resp. torsion Kisin modules killed by $p^n$ and free over $\gs_n =
\gs/p^n \gs$). Obviously $\Kfmn \subset \Kmn$ and $\bigcup_{n \geq 1}
\Kmn = \Kminf$ (the union is increasing). It is proved in theorem
proposition 2.3.2 of \cite{sec1:liu} that torsion Kisin modules are
exactly quotients of two free Kisin modules of same rank. In particular
every object in $\Kmn$ is a quotient of an object in $\Kfmn$. We finally
note that \emph{d\'evissages} with torsion Kisin modules are in general
quite easy to achieve since if $\M$ is in $\Kmn$ then $\M_{(p)} = \ker
p{|_\M}$ and $\M / \M_{(p)}$ are respectively in $\Kmn[1]$ and
$\Kmn[n-1]$ and we obviously have an exact sequence $0 \to \M_{(p)} \to
\M \to \M/\M_{(p)} \to 0$ (see Proposition 2.3.2 in \cite{sec1:liu}).

\subsection{Functors to Galois representations}

We first need to define some period rings. Let $R = \varprojlim_s
\O_{\bar K} / p$ where transition maps are Frobenius. By definition an
element $x \in R$ in a sequence $(x^{(0)}, x^{(1)}, \ldots)$ such that
$(x^{(s+1)})^p = x^{(s)}$. Fontaine proves in \cite{sec1:fontaine} that
$R$ is equipped with a valuation defined by $v_R(x) = \lim \limits_{s \to
\infty} p^s v_K(x^{(s)})$ if $x \neq 0$. (In this case, $x^{(s)}$ does
not vanish for $s$ large enough and its valuation is then well defined;
starting from this rank, the sequence $p^s v_K(x^{(s)})$ is constant.)
Note that $k$ embeds naturally in $R$ \emph{via} $\lambda \mapsto
(\lambda^{(0)}, \lambda^{(1)}, \ldots)$ where $\lambda^{(s)}$ is the
unique $p^s$-th root of $\lambda$ in $k$ (recall that $k$ is assumed to
be perfect). This embedding turns $R$ into a $k$-algebra. Now, consider
$W(R)$ (resp. $W_n(R)$) the ring of Witt vectors (resp. truncated Witt
vectors) with coefficients in $R$. It is a $W$-algebra (resp. a
$W_n(k)$-algebra). Moreover, since Frobenius is bijective on $R$,
$W_n(R) = W(R)/p^n W(R)$. Recall that we have fixed $(\pi_{s})$ a
compatible sequence of $p^s$-roots of $\pi$. It defines an element $\p
\in R$ whose Teichm\"uller representative is denoted by $[\p]$. We can
then define an embedding $\gs \inj W(R)$, $u \mapsto [\p]$. For any
positive integer $n$, reducing modulo $p^n$, we get a map $\gs_n \inj
W_n(R)$ which remains injective. In the sequel, we will often still
denote by $u$ its image in $W(R)$ and $W_n(R)$. Let $\O_\E$ be the
closure in $W(\Frac R)$ of $\gs[1/u]$ (for the $p$-adic topology).
Define $\E = \Frac \O_\E$ and $\widehat \E^\ur$ the $p$-adic completion of the
maximal (algebraic) unramified extension of $\E$ in $W(\Frac R)[1/p]$.
Denote $ \O_{\widehat\E^\ur}$ its ring of integers and put $\gs^\ur = W(R) \cap
\O_{\widehat \E^\ur}$. Clearly $\gs^\ur$ is subring of $W(R)$ and one can check
(see Proposition 2.2.1 of \cite{sec1:liu}) that it induces an embedding
$\gs^\ur_n = \gs^\ur / p^n \gs^\ur \inj W_n(R)$. Remark finally that all
previous rings are endowed with a Frobenius action.

Recall that $G$ (resp. $G_s$) is the absolute Galois group of $K$ (resp.
$K_s = K(\pi_s)$) and that $G_\infty$ is intersection of all $G_s$.
Denote by $\t{Rep}^{\t{free}}_{\Z_p} (G_\infty)$ (resp.
$\t{Rep}^{\t{tor}}_{\Z_p} (G_\infty)$) the category of free (resp. torsion)
$\Z_p$-representations of $G_\infty$. We define functors $T_\gs : \Km \to
\t{Rep}^{\t{free}}_{\Z_p} (G_\infty)$ and $T_{\gs_n} : \Kmn \to
\t{Rep}^{\t{tor}}_{\Z_p} (G_\infty)$ by:
$$T_\gs(\M):=\hom_{\gs,\varphi} (\M, \gs^\ur) \quad \text{and} \quad T_{\gs_n}(\M): =
\hom_{\gs,\varphi} (\M, \gs^\ur_n)$$
where $\hom_{\gs, \varphi}$ means that we take all $\gs$-linear morphism
that commutes with Frobenius. Note that $T_\gs(\M)$ and $T_{\gs_n} (\M)$
are \emph{not} representations of $G$ because this group does not act
trivially on $\gs \subset W(R)$. If $n' \geq n$ then any object $\M$
of $\Kmn$ is obviously also in $\Kmn[n']$ and we have a canonical
identification $T_{\gs_n} (\M) \simeq T_{\gs_{n'}} (\M)$. This fact
allows us to glue all functors $T_{\gs_n}$ and define $T_{\gs_\infty}:
\Kminf \to \t{Rep}^{\t{tor}}_{\Z_p} (G_\infty)$. An important result is the
exactness of $T_{\gs_\infty}$ (see corollary 2.3.4 of \cite{sec1:liu}).

\begin{lemma}[Fontaine]
\label{lem:tgsn}
Let $n$ be an integer and $\M$ be an object of $\Kmn$. The embedding
$\gs^\ur_n \inj W_n(R)$ induces an isomorphism $T_{\gs_n} (\M)
\ito \hom_{\gs,\varphi} (\M, W_n(R))$.
\end{lemma}

\begin{proof}
See Proposition B.1.8.3 of \cite{sec1:fontaine2}.
\end{proof}

\subsection{The modules $J_{n,c}(\M)$}

Let $n$ be an integer and $\M$ an object of $\Kmn$. For all non negative
real number $c$, we define $\a_R^{>c} = \{ x \in R \, / \, v_R(x) > c
\}$ and $[\a_R^{>c}]$ the ideal of $W_n(R)$ generated by all $[x]$ with
$x \in \a_R^{>c}$ and, by the same way, $\a_R^{\geqslant c}$ and
$[\a_R^{\geqslant c}]$. We have very explicit descriptions of these
ideals:

\begin{lemma}
\label{lem:aRc}
Let $c \in \mathbb R^+$. Then:
\begin{enumerate}
\item for all $x_0, \ldots, x_{n-1} \in R$, $(x_0, \ldots, x_{n-1})$ is
in $[\a_R^{>c}]$ (resp $[\a_R^{\geqslant c}]$) if and only if $v_R(x_i) > p^i
c$ (resp. $v_R(x_i) \geq p^i c$) for all $i$ ;
\item if $\gamma \in R$ has valuation $c$, then $[\a_R^{\geqslant c}]$ is the
principal ideal generated by $[\gamma]$.
\end{enumerate}
\end{lemma}

\begin{proof}
Easy with the formula $[z] (x_0, \ldots, x_{n-1}) = (z x_0, z^p x_1,
\ldots, z^{p^{n-1}} x_{n-1})$.
\end{proof}

Since $[\a_R^{>c}]$ is stable under $\varphi$ and $G$-action, the
quotient $W_n(R) /[\a_R^{>c}]$ inherits a Frobenius action and it makes
sense to define:
\begin{equation}
\label{eq:jna}
J_{n,c}(\M) := \hom_{\gs,\varphi} (\M, W_n(R)/[\a_R^{>c}]).
\end{equation}
It is endowed with an action of $G_\infty$. Let's also denote
$J_{n,\infty} = \hom_{\gs,\varphi} (\M, W_n(R)) \simeq T_{\gs_n}(\M)$
(Lemma \ref{lem:tgsn}). Obviously, if $c \leq c' \leq \infty$, reduction
modulo $[\a_R^{>c}]$ defines a natural $G_\infty$-equivariant morphism
$\rho_{c',c} : J_{n,c'}(\M) \to J_{n,c}(\M)$. If $c \leq c' \leq c''
\leq \infty$, we have $\rho_{c'',c} = \rho_{c',c} \circ \rho_{c'',c'}$.

\begin{lemma}
$u$ is nilpotent in $W_n[u]/E(u)^r$.
\end{lemma}

\begin{proof}
Since $E(u)$ is an Eisenstein polynomial, the congruence $E(u) \equiv
u^e \pmod p$ holds in $W[u]$. Hence $E(u)^r \equiv u^{er} \pmod p$,
which means that $u^{er}$ is divisible by $p$ in $W[u]/E(u)^r$. It
follows that $p^n$ divides $u^{ern}$ in $W[u]/E(u)^r$, \emph{i.e.}
$u^{ern}$ vanishes in $W_n[u]/E(u)^r$.
\end{proof}

Fix $N$ a positive integer such that $u^N = 0$ in $W_n[u]/E(u)^r$. By
previous proof one can take $N = ern$, but in many situations this
exponent can be improved. In the following subsection, we will examine
several examples. \emph{From now on, we put $b = \frac N {p-1}$ and $a =
b+N = \frac{pN}{p-1}$.}

\begin{prop}
\label{prop:prba}
The morphism $\rho_{\infty,b} : T_{\gs_n}(\M) \to J_{n,b }(\M)$ is
injective and its image is $\rho_{a,b }(J_{n,a}(\M))$.
\end{prop}

\begin{proof}

We first prove injectivity. Let $f : \M \to [\a_R^{>b}]$ be a
$\varphi$-morphism. We want to show that $f = 0$. First, remark that
since $\M$ is finitely generated, values of $f$ are in $[\a_R^{\geqslant
b'}]$ for some $b' > b$. Let $x \in \M$. By definition of $N$, $u^N x$
belongs to $E(u)^r \M$. By condition (\ref{eq:condition}) we can write
$u^N x = \lambda_1 \varphi(x_1) + \cdots + \lambda_k \varphi(x_k)$.
Applying $f$, we get:
$$u^N f(x) = \lambda_1 \varphi(f(x_1)) + \cdots + \lambda_k \varphi(f(x_k))
\in [\a_R^{> pb'}]$$
and then $f(x) \in [\a_R^{> pb'-N}]$ (since $u= [\p]$). Repeating the argument again and
again, we see that $f(\M) \subset \bigcap_{i \geq 0} [\a_R^{>b_i}] W_n(R)$
where $(b_i)$ is the sequence defined by $b_0 = b'$ and $b_{i+1} = p b_i
- N$. Now $b' > \frac N {p-1}$ implies $\lim\limits _{i \to \infty} b_i
= \infty$ and injectivity follows.

Let's prove the second part of the proposition. Since $\rho_{\infty,b}$
factors through $\rho_{a,b}$ we certainly have $\rho_{\infty,b }
(J_{n,\infty }) \subset \rho_{a,b }(J_{n,a})$. Conversely, we want to
prove that if $f : \M \to W_n(R)/[\a_R^{>a}]$ is a $\varphi$-morphism,
then there exists a $\varphi$-morphism (necessarily unique) $g : \M \to
W_n(R)$ such that $g \equiv f \pmod{[\a_R^{>b}]}$. Assume first $\M \in
\Kfmn$ and pick $(e_1, \ldots, e_d)$ a basis of $\M$ over $\gs_n$. Let
$A$  be a matrix with coefficients in $\gs_n$ such that:
$$(\varphi(e_1), \ldots, \varphi(e_d)) = (e_1, \ldots, e_d) A$$
and let $X$ be a line vector with coefficients in $W_n(R)$ that lifts
$(f(e_1), \ldots, f(e_d))$.

The commutation of $f$ and $\varphi$ implies the relation $XA =
\varphi(X) \pmod {[\a_R^{>a}]}$. Actually, the congruence holds in
$[\a_R^{\geqslant a'}]$ for some $a' > a$. For the rest of the proof,
fix $\alpha \in R$ some element of valuation $a'$. By Lemma
\ref{lem:aRc}.(2), $[\a_R^{\geqslant a'}]$ is the principal ideal generated
by $[\alpha]$.  Therefore, one have $ XA-\varphi(X) = [\alpha] Q$ with
coefficients of $Q$ in $W_n(R)$. We want to prove that there exists a matrix $Y$
with coefficients in $[\a_R^{>b}]$ such that $ (X + Y)A = \varphi(X +
Y)$. Let us search $Y$ of the shape $[\beta] Z$ with $\beta =
\frac{\alpha}{u^N}$ (which belongs to $R$ because of valuations) and
coefficients of $Z$ in $W_n(R)$. Our condition then becomes:
\begin{equation}
\label{eq:eqY}
[\beta] ZA = [\beta^p] \varphi(Z) + [\alpha] Q
\end{equation}
Using condition (\ref{eq:condition}) and $u^N \in E(u)^r W_n[u]$, we
find a matrix $B$ (with coefficients in $\gs_n$) such that $B A =
u^N$. Multiplying (\ref{eq:eqY}) by $B$ on the left and simplifying by
$[\alpha]$, we get the new equation:
\begin{equation}
\label{eq:eqY2}
Z = [\gamma] \varphi(Z) B  +  Q B
\end{equation}
with $\gamma = \alpha^{p-1}/{u^{pN}}$. Remark that $v_R(\gamma) =
a'(p-1) - N > 0$ ; hence $\gamma \in R$. Now define a sequence $(Z_i)$
by $Z_0 = 0$ and $Z_{i+1} = [\gamma]  \varphi(Z_i) B +  Q B$. We have
$Z_{i+1} - Z_i = [\gamma]  \varphi(Z_i - Z_{i-1}) B $. Since $v_R(\gamma)
> 0$, $Z_{i+1} - Z_i$ goes to $0$ for the $u$-adic topology (which is
separate and complete on $W_n(R)$) when $i$ goes to infinity. Hence
$(Z_i)$ converges to a limit $Z$ which is solution of (\ref{eq:eqY2}).

Finally, if $\M$ is just an object of $\Kmn$ consider $\M' \in \Kfmn$
and a surjective map $f : \M' \to \M$. Then $\ker f$ is in $\Kmn$ and
sits in the following diagram:
$$\xymatrix @R=15pt {
0 \ar[d] & & 0 \ar[d] \\
T_{\gs_n}(\M) \ar[r] \ar[d] & J_{n,a}(\M) \ar[r] \ar[d] & J_{n,b}(\M)
\ar[d] \\
T_{\gs_n}(\M') \ar[r] \ar[d] & J_{n,a}(\M') \ar[r] & J_{n,b}(\M')
\ar[d] \\
T_{\gs_n}(\ker f) \ar@{^(->}[rr] & & J_{n,b}(\ker f) }$$
where all columns are exact (by left exactness of $\hom$) and the map
on last line is injective (by first part of proposition). An easy diagram
chase then ends the proof.
\end{proof}

\begin{remark}
In general, $\rho_{a,b}$ is not surjective (nor injective) even for $a$
and $b$ big enough. Counter examples are very easy to produce: for
instance, $\M = \gs_1 \me$ equipped with $\varphi(\me) = E(u)^r \me$ is
convenient.
\end{remark}

\subsection{Brief discussion about sharpness of $N$}
\label{subsec:sharpness}

Here we are interested in finding integers $N$ (as small as possible)
such that $u^N = 0$ in $W_n[u]/E(u)^r$. As we have said before $N = ern$
is always convenient. If $n=1$, it is obviously the best constant.
However, it is not true anymore for bigger $n$: the three following
lemmas could give better exponents in many cases. We do not know how to
find the sharpest $N$ in general.

In this paragraph, we will denote by $\lceil x \rceil$ the smallest
integer not less than $x$.

\begin{lemma}
We have $u^N = 0$ in $W_n[u]/E(u)^r$ for $N = e p^{n-1} \lceil \frac r
{p^{n-1}} \rceil$.
\end{lemma}

\begin{proof}
Just remark that $E(u)^{p^{n-1}} \equiv u^{e p^{n-1}} \pmod {p^n}$.
\end{proof}

\begin{lemma}
Assume $E(u) = u^e - p$. Then $u^N = 0$ in $W_n[u]/E(u)^r$ for $N =
e(n+r-1)$.
\end{lemma}

\begin{remark}
If $K/W[1/p]$ is tamely ramified, up to changing $K$ by an unramified
extension, we can always select an uniformizer whose minimal polynomial
is $E(u) = u^e - p$.
\end{remark}

\begin{proof}
Up to performing the variables change $v = u^e$, one may assume $e=1$.
We then have an isomorphism $f : K[u]/E(u)^r \to K^r$, $P \mapsto
(P(p), P'(p), \ldots, \frac{P^{(r-1)}(p)}{(r-1)!})$ whose inverse is
given by $f^{-1}(x_0, \ldots, x_{r-1}) = x_0 + x_1 (u-p) + \cdots +
x_{r-1} (u-p)^{r-1}$. In particular $f(W[u]/E(u)^r) \supset W^r$.
Moreover:
$$f(u^N) = \left( p^N, N p^{N-1}, \ldots, \binom N {r-1} p^{N-r+1}
\right) \in p^{N-r+1} W^r = p^n W^r.$$
Conclusion follows.
\end{proof}

\begin{lemma}\label{sec1:sharp}
There exists a constant $c$ depending only on $K$ such that $u^N = 0$ in
$W_n[u]/E(u)^r$ for $N = en + c(r-1)$.
\end{lemma}

\begin{proof}
The general plan of the proof is very similar to the previous one. We
first consider the map $f : W[1/p][u]/E(u)^r \to K^r$, $P \mapsto (P(\pi),
P'(\pi), \ldots, \frac{P^{(r-1)}(\pi)}{(r-1)!})$. It is $W[1/p]$-linear
and injective. Since both sides are $W[1/p]$-vector spaces of dimension
$er$, $f$ is an isomorphism. Denote by $\varpi \in W[1/p][u]/E(u)^r$ the
preimage of $(\pi, 0, \ldots, 0)$. The inverse of $f$ is then given by
the formula:
$$f^{-1}(x_0, \ldots, x_{r-1}) = X_0(\varpi) + X_1(\varpi) (u-\varpi) +
\cdots + X_{r-1}(\varpi) (u-\varpi)^{r-1}$$
where $X_i$ are polynomials with coefficients in $K$ such that $X_i(\pi)
= x_i$.
Second, we would like to bound below the ``$p$-adic
valuation'' of $f^{-1}(x_0, \ldots, x_{r-1})$ when all $x_i$ lies in
$\O_K$. For that, we remark that $E(\varpi)$ is mapped to $0$ by
$f$; hence it vanishes. Solving this equation by successive
approximations, we find that $\varpi$ can be written $P_0(u) + P_1(u)
E(u) + \cdots + P_{r-1}(u) E(u)^{r-1}$ with $P_0(u) = u$ and:
$$E'(u) P_{i+1}(u) \equiv \frac {E(P_0(u) + P_1(u) E(u) + \cdots +
P_{i-1}(u) E(u)^{i-1})} {E(u)^i} \pmod {E(u)}$$
where $P_i$ are uniquely determined modulo $E(u)^{r-i}$. Let $F(u) \in
W[1/p][u]/E(u)$ be the inverse of $E'(u)$ and $v$ an integer such that $p^v
F(u) \in W[u]/E(u) \simeq \O_K$. (Note that $v = \lceil v_p(\D_{K/W[1/p]})
\rceil$ is convenient.) By induction we easily prove that $p^{iv} P_i(u)
\in W[u]/E(u)^{r-i}$, and then that $Q(\varpi) \in W[u]/E(u)^r$ for all
$Q \in p^{(r-1)v} W[u]$. Consequently $f(W[u]/E(u)^r) \supset
p^{(r-1)v} \O_K^r$. Finally, defining $c = ev+1$ and $N = en + c(r-1)$,
we have:
$$f(u^N) = \left( \pi^N, N \pi^{N-1}, \ldots, \binom N {r-1} \pi^{N-r+1}
\right) \in \pi^{N-r+1} \cdot \O_K^r \subset p^{(r-1)v + n} \cdot
\O_K^r$$
and we are done.
\end{proof}

\subsection{Some quotients of $W_n(R)$}

The aim of this last subsection is to study the structure of quotients
$W_n(R)/[\a_R^{>c}]$ that appears in the definition of $J_{n,c}$ (see
formula (\ref{eq:jna})). It will allow us to derive interesting
corollaries about the prolongation to a finite index subgroup of $G$ of
the natural action of $G_\infty$ on $T_{\gs_n}(\M)$.

For a non negative integer $s$, let us denote by $\theta_s$ the ring
morphism $R \to \O_{\bar K}/p$, $x = (x^{(0)}, x^{(1)}, \ldots) \mapsto
x^{(s)}$. We emphasize that it is not $k$-linear: it induces a morphism
of $k$-algebras between $R$ and $k \otimes_{k, \sigma^s} \O_{\bar K}/p$.
For a non negative real number $c$, define:
$$\a_{\bar K}^{>c} = \{x \in \bar K / v_K(x) > c\} \subset \O_{\bar
K}.$$

\begin{lemma}
\label{lem:rz}
Let $c$ be a positive real number. For any integer $s > \log_p (\frac
c e)$, the map $\theta_s$ induces a Galois equivariant isomorphism of
$k$-algebras
$$R/\a_R^{>c} \to k \otimes_{k, \sigma^s} \O_{\bar K}/
\a_{\bar K}^{>c/p^s}.$$
\end{lemma}

\begin{proof}
The map is clearly surjective. It remains to show that
$x = (x^{(0)}, x^{(1)}, \ldots)$ has valuation greater than $c$ if
and only if $v_K(x^{(s)}) > \frac c {p^s}$, which follows directly
from $\frac c{p^s} < e$.
\end{proof}

\begin{prop}
\label{prop:wnrz}
Let $c$ be a positive real number. For any $s > n-1 + \log_p (\frac c
e)$, $\theta_s$ induces a Galois equivariant isomorphism of
$W_n(k)$-algebras:
$$\frac{W_n(R)}{[\a_R^{>c}]} \to W_n(k) \otimes_{W_n(k), \sigma^s}
\frac{W_n(\O_{\bar K}/p)}{ [\a_{\bar K}^{>c/p^s}] }.$$
\end{prop}

\begin{proof}
Since $\theta_s$ is surjective,  the map considered in the lemma is also
surjective. Let $x = (x_0, \ldots, x_{n-1}) \in W_n(R)$ and assume that
$x^{(s)} = (x_0^{(s)}, \ldots, x_{n-1}^{(s)})$ lies in $[\a_{\bar K}
^{>c/p^s}]$. By an analogue of Lemma \ref{lem:aRc}.(1), one obtain
$v_K(x_i^{(s)}) > \frac c {p^{i-s}}$ for all $i$. Hence, $x_i^{(s)}$ is in
$\a_{\bar K}^{>cp^i/p^s}$. Since $\log_p(\frac{c p^i}{e}) = i + \log_p(\frac{c}{e}) \leq
n-1 + \log_p(\frac{c}{e}) < s$, we can apply Lemma \ref{lem:rz} and deduce $x_i
\in [\a_R^{>c p^i}]$, \emph{i.e.} $v_R(x_i) > c p^{n-1+i}$ for all $i$.
By Lemma \ref{lem:aRc}.(1), it follows that $x \in [\a_R^{>c}]$. Thus,
the map of the proposition is injective and we are done.
\end{proof}

Define increasing functions $s_0$ and $s_1$ by $s_0(c) = n-1 +
\log_p(\frac c e)$ and $s_1 (c) = n - 1 + \log_p(\frac {c(p-1)}{ep}) =
s_0(c) + \log_p(1 - \frac 1 p)$. Recall that we have defined $a = \frac
{pN} {p-1}$ (where $N$ is an integer such that $u^N = 0$ in
$W_n[u]/E(u)$) and set finally $\smin = s_1(a) = n - 1 + \log_p(\frac N
e)$. If we choose $N = ern$, we just have $\smin = n-1 + \log_p(rn)$.

\begin{prop}
\label{prop:prolong}
Let $n$ be a positive integer and $\M \in \Kmn$. For any non negative
integer $s > s_1 (c)$, the  natural action of $G_s$ on $W_n(R)$ turns
$J_{n,c}(\M)$ into a $\Z_p[G_s]$-module. Furthermore, we have the
following compatibilities:
\begin{itemize}
\item the action of $G_s$ is compatible with the usual action of
$G_\infty$ on $J_{n,c}(\M)$;
\item if $s' \geq s \geq s_1(c)$, actions of $G_{s'}$ and $G_s$ on
$J_{n,c}$ are compatible each other;
\item if $c' \geq c$ and $s \geq s_1(c')$, then $\rho_{c',c} :
J_{n,c'}(\M) \to J_{n,c}(\M)$ is $G_s$-equivariant.
\end{itemize}
\end{prop}

\begin{proof}
For the first statement, it is enough to show that $G_s$ acts trivially
on $u \in W_n(R)/[\a_R^{>c}]$ for $s = 1 + [s_1(c)]$ (where $[\cdot]$
denotes the integer part). Put $s' = 1 + [s_0(c)]$. Since $0 \leq
\log_p(\frac p {p-1}) \leq 1$, we have $s' = s$ or $s' = s + 1$. By
Proposition \ref{prop:wnrz}, $W_n(R) / [\a_R^{>c}]$ is isomorphic to
$W_n(\O_{\bar K}/p)/[\a_{\bar K}^{>c/p^{s'}}]$. Hence we have to show
that $g [\pi_{s'}] - [\pi_{s'}]$ belongs to $[\a_R^{>c/p^{s'}}]$ for all
$g \in G_s$. It is clear for $g \in G_{s'}$ (since the difference
vanishes). It remains to consider the case where $s' = s+1$ and $g
\not\in G_{s'} = G_{s+1}$. Then $g \pi_{s+1} = (1 + \eta) \pi_{s+1}$
where $(1 + \eta)$ is a primitive $p$-th root of unity. Let us compute
$(g \pi_{s+1}, 0, \ldots, 0) - (\pi_{s+1}, 0, \ldots, 0) = (x_0, \ldots,
x_{n-1})$ in $W_n(\O_{\bar K})$. By writing phantom components, we get
the following system:
$$\left\{ \begin{array}{l}
x_0 = \eta \: \pi_{s+1} \\
x_0^p + p x_1 = 0 \\
\hspace{1em} \vdots \\
x_0^{p^{n-1}} + p x_1^{p^{n-2}} + \cdots + p^{n-1} x_{n-1} = 0 \\
\end{array} \right.$$
Using $v_K(\eta) = \frac e{p-1}$, we easily prove by induction on $i$
that $v_K(x_i) = \frac e{p-1} + \frac 1{p^{s+1-i}}$. Thus $v_K(x_i) >
\frac c {p^{s+1-i}}$ for all $i$ and $(x_0, \ldots, x_{n-1}) \in
 [\a_{\bar K}^{>c/p^{s'}}]$ as expected.

Second part of proposition (\emph{i.e.} compatibilities) is obvious.
\end{proof}

\begin{remark}
\label{rem:slimp}
If $c \geq \frac{p-1}{p-2}$, the bound $s_1(c)$ that appears in the
Theorem can be replaced by $s_1(c-1)$. The proof is totally the same.
\end{remark}

\begin{theorem}
\label{theo:prolong}
For any $\M \in \Kmn$ and any integer $s > \smin$, $T_{\gs_n} (\M)$ is
canonically endowed with an action of $G_s$ (which prolongs the natural
action of $G_\infty$).
\end{theorem}

\begin{proof}
Just combine Propositions \ref{prop:prba} and Proposition \ref{prop:prolong}.
\end{proof}

\begin{remark}
Using Remark \ref{rem:slimp}, it appears that we may replace $\smin
= s_1(a)$ by $s_1(a-1)$ in previous Theorem. However, it won't be
useful in the sequel since $\smin$ is really needed in Theorem
\ref{compatible}.
\end{remark}

\section{Torsion semi-stable Galois representations}

In this section, we use the theory of $(\varphi, \hat G)$-modules to
define $\hat J_{n, a} (\hM)$ attached to $p^n$-torsion semi-stable
representation $T$. After establishing isomorphism (of $\Z_p
[G_\infty]$-modules) between $\hat J _{n, a }(\hM)$ and $J_{n ,a}(\M)$,
we will show that ${J}_{n , a} (\M) \simeq T $ as $G_{s}$-modules with $s
> \smin$.

\subsection{Torsion $(\varphi, \hat G)$-modules}

We refer readers to \cite{fontaine3} for the definition and standard
facts on semi-stable representations.

We first review some facts on $(\varphi, \hat G)$-modules in \cite{liu4}
and extend them to $p^n$-torsion case. We denote by $S$ the $p$-adic
completion of the divided power envelope of $W(k)[u]$ with respect to
the ideal generated by $E(u)$. There is a unique continuous map
(Frobenius) $\varphi: S \to S$ which extends the Frobenius on $\gs$.
Define a continuous $W(k)$-linear derivation $N: S \to S$ such that
$N(u)= -u$.

Recall $R= \varprojlim_s \O_{\bar K }/p$.  There is a unique surjective
continuous map $\theta: W(R) \to \widehat{\O}_{\bar K }$ which lifts the
projection $R \to \O_{\bar K }/p$ onto the first factor in the inverse
limit. We denote by $A_{\t {cris}}$ the $p$-adic completion of the
divided power envelope of $W(R)$ with respect to $\t{Ker}(\theta)$.
Recall that $[\p] \in W(R)$ is the Teichm\"uller representative of $\p =
(\pi_s)_{s\geq 0} \in R$ and we embed the $W(k)$-algebra $W(k)[u]$ into
$W(R)$ via $u \mapsto [\p]$. Since $\theta(\p)= \pi$, this embedding
extends to an embedding $\gs \inj S \inj A_{\t{cris}}$, and $\theta |_S$
is the $W(k)$-linear map $s: S \to \O_K $ defined by sending $u$ to
$\pi$. The embedding is compatible with Frobenius endomorphisms. As
usual, we write $B_\t{cris}^+ := A_\t{cris}[1/p]$.

For any field extension $F/\Q_p$, set $F_{p^\infty} := \bigcup
\limits_{n=1}^\infty F (\zeta_ {p^n})$ with $\zeta _{p^n}$ a primitive
$p^n$-th root of unity. Note that $K_{\infty, p^\infty} := \bigcup \limits
_{n=1}^\infty K({\pi_n}, \zeta_{p^n})$ is Galois over $K$. Let
$G_{p^\infty} := \gal(K_{\infty, p^\infty}/ K_{p^\infty})$, $H_K:= \gal
(K_{\infty, p^\infty}/ K_\infty)$ and $\hat G: =\gal (K_{\infty,
p^\infty}/K) $. By Lemma 5.1.2 in \cite{liu3}, we have $K_{p^\infty}
\cap K_\infty = K$, $\hat G= G_{p^\infty} \rtimes H_K$ and $G_{p^\infty}
\simeq \Z_p(1)$. For any $g \in G$, write $\e(g)= g(\p)/{\p}$. Then
$\e(g)$ is a cocycle from $G$ to the group of units of $R^*$. In
particular, fixing a topological generator $\tau$ of $G_{p^\infty}$, the
fact that $\hat G= G_{p^\infty} \rtimes H_K$ implies that $\e(\tau)=
(\epsilon_s)_{s \geq 0} \in R^*$ with $\epsilon_s$ a \emph{primitive}
$p^s$-th root of unity. Therefore, $t := -\log([\e(\tau)]) \in \acris$ is
well defined and for any $g \in G$, $g(t) = \chi(g)t$ where $\chi$ is
the cyclotomic character. We reserve $\e$ for $\e(\tau)$.


For any integer $n\geq 0$, let $t^{\{n\}}= t ^{r(n)} \gamma_{\tilde
q(n)}(t^{p-1}/p)$ where $n = (p-1)\tilde q(n) + r(n)$ with $ 0 \leq
r(n)< p-1$ and $\gamma_i(x) = \frac{x^i}{i!}$ is the standard divided
power. Define subrings $\R_{K_0}$ and $\hR$ of $ B^+_\cris$ as in
\S2.2, \cite{liu4}:
$$\R_{K_0}: =\left\{x = \sum_{i=0 }^\infty f_i t^{\{i\}}, f_i \in S[1/p] \t{
and } f_i \to 0\t{ as }i \to +\infty \right\}$$
and $\hR:= W(R) \cap \R_{K_0}$.
Let $I_+ R= \{x \in R \, / \, v_R (x) >0\} = \a_R ^{>0}$ be the maximal ideal of $R$.
We have exact sequences
$$0\to W_n (I_+ R) \to W_n(R)\overset{\nu_n}{\to} W(\bar k) \to 0 \t{ and } 0
\to W(I_+R) \to W(R) \overset{\nu}{\to} W(\bar k) \to 0. $$
One can naturally extend $\nu$ to $\nu : B^+_\cris \to W(\bar
k)[\frac{1}{p}]$ (see the proof of Lemma 2.2.1 in \cite{liu4}). For any
subring $A$ of $B^+_\cris$ (resp. $W_n(R)$), we write $I_+A =
\t{Ker}(\nu) \cap A$ (resp. $I_+A = \t{Ker}(\nu_n) \cap A$) and $I_+:=
I_+ \hR$. Now recall $M_n$ stands for $M/p^nM$.

\begin{lemma}
\label{torsion}
We have the following commutative diagram :
\begin{equation}
\begin{split}
\xymatrix{ 0 \ar[r] & W_n (I_+R) \ar[r] & W_n(R) \ar[r]^- {\nu_n} &
W_n(\bar k) \ar[r] & 0 \\
0 \ar[r] & I_{+, n}\ar@{^{(}->}[u] \ar[r] & \hR_n \ar[r]\ar@{^{(}->}[u]
& W_n(k) \ar[r]\ar@{^{(}->}[u] & 0}
\end{split}
\end{equation}
such that both rows are short exact and all vertical arrows are injective.
\end{lemma}

\begin{proof}
By Lemma 2.2.1 in \cite{liu4}, we have a commutative diagram of exact
sequences:
$$ \xymatrix{ 0 \ar[r] & I_+W(R) \ar[r] & W(R) \ar[r]^- {\nu} & W(\bar
k) \ar[r] & 0 \\
0 \ar[r] & I_{+}\ar@{^{(}->}[u] \ar[r] & \hR \ar[r]\ar@{^{(}->}[u] &
W(k) \ar[r]\ar@{^{(}->}[u] & 0}$$
Modulo $p^n$ and noting that $I_+ W(R)= I_+B^+_\cris \cap W(R)= W
(I_+R)$, we get
$$ \xymatrix{ 0 \ar[r] & W_n (I_+R) \ar[r] & W_n(R) \ar[r]^- {\nu_n} &
W_n(\bar k) \ar[r] & 0 \\
& I_{+, n}\ar[u] \ar[r] & \hR_n \ar[r]\ar[u] & W_n(k) \ar[r]\ar[u] & 0}$$
Now it suffices to show that the bottom arrow is left exact and the last
two vertical arrows are injective. The last one is
obvious. To see the middle arrow, it suffices to show that $(p^n W(R))
\cap \hR = p^n \hR$.
Note that $\hR = \R_{K_0} \cap W(R)$. Let $x \in W(R)$ such that $p^n x
\in \hR\subset \R_{K_0}$. Then $x \in W(R) \cap \R_{K_0} = \hR$. So
$p^n x \in p^n \hR$ and $(p^n W(R)) \cap \hR = p^n \hR$. To see the
bottom is left exact, it suffices to show that $I_+ \cap p^n \hR=p ^n
I_+$. But $I_+ = I_{+}\R_{K_0} \cap \hR$. Let $ x \in \hR$ such that
$p^nx \in I_+ \cap p^n \hR$. Then $x \in I_{+}\R_{K_0}$. Thus $x \in
I_{+}\R_{K_0} \cap \hR = I_+$ and $I_+ \cap p^n \hR=p ^n I_+$.
\end{proof}

As in Lemma 2.2.1 in \cite{liu4}, we see that $\hR$ (resp.  $\hR_n$) is
$\varphi$-stable and $G$-stable subring of $W(R)$ (resp. $W_n (R)$),
$G$-action on $\hR$ factors through $\hat G $. Let $(\M, \varphi)$ be a
finite free or $p^n$-torsion Kisin module of height $\leq r$, set $\hM:= \hR
\otimes_{\varphi, \gs} \M$ and consider the following map
\begin{equation}
\label{submodule}
\M \simeq \gs \otimes_{\gs} \M \to \gs \otimes _{\varphi, \gs}\M \to
\hR\otimes_{\varphi, \gs} \M = \hM.
\end{equation}
We claim the it is an injective (thus $\M$ can be always regarded
as a $\varphi(\gs)$-submodule of $\hM$). Indeed, by Lemma \ref{torsion},
we have $\varphi(\gs_n) \inj \hR_n \inj W_n(R)$. Thus the claim is clear
if $\M$ is finite $\gs$-free or $\M$ is finite $\gs_n$-free. For a
general $\M$ which is killed by $p^n$, by the discussion in the end of
\S2.1, $\M$ can be written as a successive extension of finite free
$\gs_1$-modules. Therefore one can reduce the proof of the claim to
the following lemma.

\begin{lemma}
\label{exact}
The functor $\M \mapsto \hR \otimes_{ \varphi, \gs}\M $ (resp. $\M
\mapsto W(R) \otimes_{\varphi, \gs}\M$) is an exact functor from the
category of Kisin modules to the category of $\hR$-modules (resp.
$W(R)$-modules).
\end{lemma}

\begin{proof}
We only prove the exactness of the first functor, the proof for the
second being totally the same. It suffices to prove that $\t{Tor}_1 ^\gs
(\M , \hR)= 0$ for any Kisin module $\M$. Note that there exists finite
free Kisin modules $\L_1 \subset \L_2$ such that $\M= \L_2/\L_1$ (\emph{cf}
discussion in the end of \S2.1). Since $\hR\inj W(R)$ is an integral
domain and $\varphi: W(R) \to W(R)$ is injective, we see $\hR
\otimes_{\varphi, \gs} \L_1 \to \hR \otimes_{\varphi , \gs} \L_2$ is
injective. Thus $\t{Tor}_1 ^\gs (\M , \hR)= 0$.
\end{proof}

Let $(\M, \varphi)$ be a Kisin module of height $\leq r$ and $\hM:= \hR
\otimes_{\gs, \varphi}\M$. Frobenius $\varphi$ on $\M$ can be extended
to $\hM$ semi-linearly by $\varphi_{\hM}(a \otimes x) = \varphi_{\hR}(a)
\otimes \varphi_{\M}(x)$.

Now we can make the following definition: a \emph{$(\varphi, \hat
G)$-module of height $\leq r$} is a triple $(\M , \varphi, \hat G)$ where
\begin{enumerate}
\item $(\M, \varphi_\M)$ is a Kisin module of height $\leq r$;
\item $\hat G$ is a $\hR$-semi-linear $\hat G$-action on $\hat \M=\hR
\otimes_{\varphi, \gs} \M$;
\item $\hat G$ commutes with $\varphi_{\hM}$ on $\hM$, \emph{i.e.} for
any $g \in \hat G$, $g \varphi_{\hM} = \varphi_{\hM} g$;
\item regard $\M$ as a $\varphi(\gs)$-submodule in $ \hM $, then $\M
\subset \hM ^{H_K}$;
\item $\hat G$ acts on $W(k)$-module $M:= \hM/I_+\hM\simeq \M/u\M$ trivially.
\end{enumerate}

A morphism $f: (\M , \varphi,\hat G) \to (\M', \varphi', \hat G')$ is a
morphism $\f: (\M, \varphi) \to (\M' , \varphi')$ of Kisin modules such
that $\hR \otimes_{\varphi, \gs} \f: \hM \to \hM '$ is $\hat
G$-equivariant. If $\hM = (\M , \varphi, \hat G)$ be a $(\varphi, \hat
G)$-module, we will often abuse notations by denoting $\hM$ the
underline module $\hR \otimes_{\varphi, \gs}\M$. A $(\varphi, \hat
G)$-module $\hM := (\M, \varphi, \hat G)$ is called \emph{finite free}
(resp. $p^n$-torsion) if $\M$ is finite $\gs$-free (resp. $\M$ is killed
by $p^n$).

Let $\hM=(\M, \varphi, \hat G)$ be a $(\varphi, \hat G)$-module. We can
associate $\Z_p[G]$-modules:
\begin{equation*}\label{hatT}
\hat T (\hM) := \t{Hom}_{\hR, \varphi}( \hM, W(R)) \t{ if $\M$ is finite
$\gs$-free.}
\end{equation*}
and
\begin{equation*}\label{hatTn}
\hat T_n (\hM) := \t{Hom}_{\hR, \varphi}(\hM, W_n (R)) \t{ if $\M$ is of
$p^n$-torsion.}
\end{equation*}
Here $G$ acts on $\hat T(\hM)$ (resp. $\hat T_n(\hM)$) via $g (f)(x) = g
(f(g^{-1}(x)))$ for any $g \in G$ and $f \in \hat T(\hM)$ (resp. $\hat
T_n (\hM)$). For any $f \in T_\gs(\M) $ (resp. $T_{\gs_n} (\M)$), set
$\theta (f) \in \t{Hom}_{\hR}(\hM, W(R))$ (resp.  $\theta_n (f) \in
\t{Hom}_{\hR}(\hM, W_n (R))$) via:
\begin{equation}\label{theta}
\theta(f) (a\otimes x) \ (\t{resp. }\theta_n (f)(a \otimes x))= a
\varphi (f(x)) \t{ for any } a \in \hR, x \in \M.
\end{equation}

It is routine to check that $\theta: T_\gs(\M) \to \hat T(\hM)$ (resp.
$\theta_n: T_{\gs_n}(\M) \to \hat T_n(\hM)$) is well-defined.

Denote by $\t{Rep}_\t{tor}(G)$ the category of $G$-representations on
finite type $\Z_p$-modules which are killed by some $p$-power, and
$\t{Rep}_\t{tor}^{\t{ss}, r }(G)$ the full subcategory of \emph{torsion
semi-stable representations} with Hodge-Tate weights in $\{0, \dots ,
r\}$ in the sense that there exist $G$-stable $\Z_p$-lattices $L_1
\subset L_2 \subset V$ such that $V$ is semi-stable with Hodge-Tate
weights in $\{0, \dots, r\}$ and $T \simeq L_2/L_1$ as $\Z_p
[G]$-modules. The following is the main result of this subsection.

\begin{theorem}\label{main}
\begin{enumerate}
\item Let $\hM:=(\M, \varphi, \hat G) $ be a $(\varphi, \hat G)$-module.
Then $\theta$ (resp. $\theta_n$) induces a natural isomorphism of
$\Z_p[G_\infty]$-modules $\theta: T_\gs(\M) \ito \hat T (\hM)$ (resp.
$\theta_n: T_{\gs_n}(\M) \ito \hat T_n (\hM)$).
\item $\hat T$ induces an anti-equivalence between the category of
finite free $(\varphi, \hat G)$-modules of height $\leq r$ and the category
of $G$-stable $\Z_p$-lattices in semi-stable representations with
Hodge-Tate weights in $\{0, \dots, r\}$.
\item For any $T \in \trep$, there exists a torsion $(\varphi, \hat
G)$-modules $\hM$ such that $\hat T_n (\hM) \simeq T$ as $\Z_p
[G]$-modules.
\end{enumerate}
\end{theorem}

\begin{proof}
(1) If $\M$ is finite $\gs$-free then it has been proved in Theorem
2.3.1 in \cite{liu4}. The proof of the $p^n$-torsion case is almost the
same, except one need to check that $\M$ is a $\varphi(\gs)$-submodule
of $\hM$ via \eqref{submodule}, which has been proved below
\eqref{submodule}.

(2) See Theorem 2.3.1 in \cite{liu4}.

(3) Let $L_1 \subset L_2$ be $G$-stable $\Z_p$-lattices inside a
semi-stable representation with Hodge-Tate weights in $\{0, \dots , r\}$
such that $T \simeq L_2/L_1$ as $\Z_p[G]$-modules. By (2), there exists
an injection of Kisin modules (resp. $(\varphi, \hat G)$-modules) $i:
\L_2 \inj \L_1$ (resp. $\hat i : \hat \L_2 \inj \hat \L_1$) that
corresponds the inclusion $L_1 \subset L_2$. Write $\M : = \L_1/ \L_2$
(resp. $\hM: = \hat \L_1/ \hat \L_2$). Apparently, there are a
$\varphi$-action and a $\hat G$-action on $\hM$ induced from $\hat \L_1$
and $\hat \L_2$. We claim that $\hM \simeq \hR \otimes _ {\varphi,
\gs}\M$ as $\varphi$-modules and $(\M, \varphi, \hat G)$ is a $(\varphi, \hat G)$-modules. To see these, tensor $\hR$ to the exact
sequence $0 \to \L_2 \to \L_1 \to \M \to 0$. By the proof of Lemma
\ref{exact}, we see that the sequence $0 \to \hat \L_2 \to \hat \L_1 \to
\hR \otimes_{\varphi, \gs}\M \to 0 $ is still exact. Thus $\hM \simeq
\hR \otimes_{\varphi ,\gs}\M$ as $\varphi$-modules. Moveover, we have
the following commutative diagram
$$\xymatrix{ 0 \ar[r] & \hat \L_2 \ar[r] & \hat \L_1 \ar[r] & \hR
\otimes_{\varphi, \gs}\M \ar[r]& 0\\
0 \ar[r] & \L_2 \ar[r] \ar@{^{(}->}[u] & \L_1 \ar[r] \ar@{^{(}->}[u] &
\M \ar[r] \ar@{^{(}->}[u] & 0 }$$
So $\varphi$-action and $\hat G$-action on $\hM$ commutates, $H_K $ acts
on $\M$ (as $\varphi(\gs)$-submodule in \eqref{submodule}) trivially,
and $\hat G$ acts on $\hM/ I_+\hM$ trivially. Thus $\hM= (\M , \varphi,
\hat G)$ is a $(\varphi, \hat G)$-module. Finally, to see that $\hat
T_n(\hM) \simeq T$ as $\Z_p[G]$-modules, it suffices to show that $\hat
T_n(\hM) \simeq \hat T(\hat \L_2)/ \hat T(\hat \L_1)$ and we reduce the
proof to the following Lemma.
\end{proof}

\begin{lemma}\label{free-torsion}
Let $0 \to \hL_2 \to \hL_1 \to \hM \to 0 $ be an exact sequence of
$(\varphi, G)$-modules with $\hL_1$, $\hL_2$ finite free and $\hM$
killed by $p^n$. Then we have an exact sequence of $\Z_p[G]$-modules $0
\to \hat T(\hL_1) \to \hat T(\hL_2) \to \hat T_n(\hM) \to 0$.
\end{lemma}

\begin{proof} Consider the following commutative diagram:
$$\xymatrix { 0 \ar[r] & p^n\hat \L_2 \ar@{^{(}->}[d] \ar[r] & p^n \hat
\L_1 \ar[r] \ar@{^{(}->}[d] & \hM \ar[d]\ar[r]& 0 \\
0 \ar[r] & \hat \L_2 \ar[r] & \hat \L_1 \ar[r] & \hM \ar[r]& 0} $$
where the last vertical map is $p^n=0$. By Snake lemma, we have an exact
sequence
\begin{equation}\label{snake}
0 \to \hM \to (\hL_2)_n \to (\hL _1)_n \to \hM \to 0 .
\end{equation}
Then we get a sequence of $\Z_p[G]$-modules
\begin{equation}\label{snake2}
0 \to \hat T_n(\hM) \to\hat T_n( (\hL_1)_n) \to \hat T_n((\hL _2)_n) \to
\hat T_n (\hM ) \to 0.
\end{equation}
Since $\hat T( (\hL_i)_n) \simeq (\hat T (\hL_i))_n$ for $i=1,2$, it
suffices to show that the above sequence is exact. But the underline
Kisin modules of the exact sequence \eqref{snake} is exact. Since
$T_\gs$ is exact, we get an exact sequence
$$0 \to T_{\gs_n }(\M) \to T_{\gs_n}( (\L_1)_n) \to T_{\gs_n}((\L
_2)_n) \to T_{\gs_n}(\M ) \to 0.$$
Now the exactness of \eqref{snake2} follows from Theorem \ref{main}.(1).
\end{proof}

\begin{remark}\label{nonunqiue}
For a fixed $T \in \trep$, it may exist two different $(\varphi, \hat
G)$-modules $\hM$, $\hM'$ such that $\hat T_{n} (\hM) \simeq \hat T_{n}
(\hM ') \simeq T$. The classical example of this is that $T=\Z/p\Z$ with
the trivial $G$-action and $K = \Q_p(\zeta_p)$.
\end{remark}

\subsection{$G_s $-action on  $\hat T(\hL)$}

Let $T \in \trep$ be a $p^n$-torsion representation, and $T \simeq
L'/L$ where $L\subset L' $ are $G$-stable $\Z_p$-lattices in a
semi-stable representation $V$ with Hodge-Tate weights in $\{0, \dots ,
r\}$. By Theorem \ref{main}, there exists $(\varphi, \hat G)$-modules
$\hL '\inj \hL$ such that $\hat T(\hL) \inj \hat T(\hL')$ corresponds
to the injection $L\subset L'$ and $\hat T_n (\hM)\simeq T$ where $\hM
:= \hL /\hL'$. Now write $\L$, $\L'$, $\M$ the underline Kisin modules
for $\hL$, $\hL'$, $\hM$ respectively. Set $\D := S[1/p] \otimes_{\varphi, \gs}
\L$ and recall $\e(g):= \frac{g (\p)}{\p}$ for any $g \in G$. \S3.2 of
\cite{liu4} explains that there exists an unique $W(k)$-linear
differential operator $N : \D \to \D$ over $N(u)= -u$ such that $G$ acts
on $B^+_\cris \otimes_{S}\D\simeq B^+_\cris \otimes_{\hR}\hL$ via
\begin{equation}\label{action}
g(a \otimes x)= \sum_{i=0}^\infty g(a)\gamma _i (-\log([\e(g)]))
\otimes N^i (x), \ \t{ for any } a \in B^+_\cris,  \  x \in \D.
\end{equation}

In particular, recall $t:= -\log ([\e])$ with $\e= \e(\tau)$ and $\tau$
is a fixed generator in $G_{p^\infty}$.  For any $x \in \L$, we have
$\tau(x) = \sum\limits_{i=0}^\infty \gamma_i (t)\otimes N^i (x)$. Let $A
\subset B^+_\cris$ be a $\varphi$-stable subring. Set
$$I^{[m]}A =\{a \in A| \varphi^n( a) \in A\cap \t{Fil}^m B^+_\cris, \t{
for all } n \geq 0 \}.$$
By proposition 5.1.3 in \cite{sec1:fontaine}, $I^{[m]}W(R)$ is generated
by $([\e]-1)^m $ and $v_R (\e -1)= \frac{ep}{p-1}$.

Now, define $s_2(c) := n-1 + \log_p (\frac{(p-1)c}{e}) = s_1(c)+1$. We
have the following lemma:


\begin{lemma}
\label{push}
For any $s > s_2(c)$, $g \in G_s$, and $ x\in \L$
$$g(x)-x \in ([\a_R^{>pc}] + p^n W(R))\otimes_{\hR} \hL$$
where, in a slight abuse of notations, we still denote by $[\a_R^{>c}]$
the ideal of $W(R)$ generated by all $[x]$ with $x \in \a_R^{>c}$.
\end{lemma}

\begin{proof}
Note that the $G$-action on $\hL$ factors through $\hat G$. So it
suffices to consider the action of $\hat G_s$, which is the image of
$G_s$ in $\hat G$. By Lemma 5.1.2 of \cite{liu3} applied to $K_s$, we
see that $\hat G_s :=G_{s, p^\infty} \rtimes H_K $, where $G_{s,
p^\infty}=\gal(K_{\infty, p^\infty}/ K_{s, p^\infty})$. Note that $H_K$
acts on $\L$ trivially and $G_{s, p^\infty}$ is topologically generated
by $\tau ^{p^s}$. Thus it suffices to prove the proposition for $g =
\tau ^{p^s}$. Writing
$$\tau ^{p^s}-1 = \sum\limits_{i=1}^{p^s}{p^s
\choose i} (\tau -1)^i = \sum_{i=1}^{p^s}\frac{p^s} i {p^s-1
\choose i-1} (\tau -1)^i$$
we see that it is enough to show
$$(\tau-1)^i (x) \in ([\a_R^{>pc}] + p^{n-s+v_p(i)} W(R))
\otimes_{\hR} \hL$$
for all integer $i$ such that $v_p(i) > s-n$, \emph{i.e.} $v_p(i)
\geq s-n+1$.

Using formula \eqref{action}, an easy induction on $l$ shows that
\begin{equation}
\label{n-1}
(\tau -1)^l (x) = \sum_{m=l}^ \infty\left (\sum_{i_1 + \cdots +i_l= m,
i_j \geq 1 } \frac{m!}{i_1! \cdots i_l ! }\right ) \gamma_{m}(t) \otimes
N^m(x)
\end{equation}
for any $l \geq 0 $ and $x \in \D$.
In particular, $(\tau -1)^l (x) \in ({I}^{[l]} B^+_\cris)
(B^+_{\cris}\otimes_S \D)$. Since $x\in \L$ and $W(R)\otimes_{\hR}\hL$
is $G$-stable, we get $(\tau-1)^l(x) \in (I^{[l]}W(R)) (W(R)
\otimes_{\hR}\hL) $. So it suffices to show that $([\e]-1)^i \in [\a_R ^{> pc}] + p^{n-s+v_p(i)}W(R)$ for any $i$ satisfying $v_p (i) \geq s-n+1$.
Write $i = p^{v+s-n+1} m$ with $v \geq 0$ and $p \nmid m$. From
$v_R(\e-1) = \frac{ep}{p-1}$, it follows that $([\e]-1)^{p^{s-n+1}}
\in [\a_R^{>pc}] + p W(R)$. By induction (on $v$), we easily find
$([\e]-1)^{p^{v+s-n+1}} \in [\a_R^{>pc}] + p^{v+1} W(R)$, which is
exactly the expected result.
\end{proof}

\subsection{Comparison between $\hat J_{n,c}(\hat {\mathfrak M})$ and
$J_{n,c} (\mathfrak M)$}

Let $T$ be a $p^n$-torsion semi-stable representation and $\hM $ an
attached $(\varphi, \hat G)$-module \emph{via} Theorem \ref{main}.(3). We have
the following definitions and results similar to \S 2. For any non
negative real number $c$, set
$$\hat J_{n,c} (\hM):= \t{Hom}_{\hR_n, \varphi} (\hM,
W_n(R)/[\a_R^{>pc}]).$$
and $\hat J_{n,\infty} (\hM) =\hat T_n ( \hM)$.

For any $c \leq \infty$, $\hat J_{n , c}(\hM)$ is a $\Z_p[G]$-module
and, for any $c \leq c' \leq \infty$, canonical projection induces a map
$\hat \rho_{c',c}: \hat J_{n,c'} (\hM)\to \hat J_{n,c}(\hM)$. Moreover,
we have a morphism $\theta_{n, c}(f) \in \t{Hom}_{\hR_n}(\hM,
W_n(R)/[\a_R^{>{pc}}])$ defined by:
\begin{equation}\label{thetana}
(\forall \alpha \in \hR) \, (\forall x \in \M) \quad
\theta_{n, c}(f)(\alpha \otimes x) = \alpha \varphi(f (x)).
\end{equation}

\begin{prop}
\label{forcompatible}
For any non negative integer $s > s_2(c) = n-1 + \log_p (\frac{(p-1)c}{e})$,
$\theta_{n, c} : J_{n, c} (\M) \to \hat J_{n, c} (\hM)$ is an
isomorphism of $\Z_p [G_s]$-modules.
\end{prop}

\begin{remark}
Since $s_2(c) = s_1(c)+1 \geq s_1(c)$, Proposition \ref{prop:prolong}
shows that $J_{n, c} (\M)$ is endowed with an action of $G_s$. Hence it
makes sense to claim that $\theta_{n,c}$ is $G_s$-equivariant. We do not
know if Proposition \ref{forcompatible} remains true under the
smaller assumption ``$s > s_1(c)$'': we conjecture that it is false but
we do not know any counter-example.
\end{remark}

\begin{proof}
It is routine to check that $\theta_{n, c}(f)$ is well defined and
preserves Frobenius. Hence $\theta_{n,c}$ is also well defined. Let's
first prove that it is bijective. Remark that $\varphi:
W_n(R)/[\a_R^{>c}] \ito W_n (R)/[\a_R^{>pc}]$ is an isomorphism. It
follows easily that $\theta _{n, c}$ is injective. For any $\hat f \in
\hat J_{n, c}(\hM)$, set $f' : = \hat f|_{\M}$ (recall that we regard
$\M$ as a $\varphi(\gs)$-submodule of $\hM$ via \eqref{submodule}). Then
$f' : \M \to W_n(R)/[\a_R^{>pc}]$ is $\varphi(\gs)$-linear map and is
compatible with Frobenius. Since $\varphi: W_n(R)/[\a_R^{>c}] \simeq W_n
(R)/[\a_R^{>pc}]$, we can set $f = \varphi^{-1} (f ') : \M \to
W_n(R)/[\a_R^{>c}]$. It is finally easy to check that $f$ belongs to
$J_{n,c}(\M)$ and that $\theta_{n,c} (f) =\hat f$. Hence $\theta_{n,c}$
is surjective, as required.

It remains to prove that $\theta_{n,c}$ is $G_s$-equivariant. Let $g \in
G_s$, $\alpha \in \hR$ and $x \in \M$. Expanding the definitions, we get
$g (\theta_{n,c}(f))(\alpha \otimes x) = \alpha g(\theta_{n , c}(f)(g
^{-1}(1 \otimes x)))$. Moreover Lemma \ref{push} shows that
$g^{-1}(1\otimes x)$ is congruent to $1\otimes x$ modulo $[\a_R^{>pc}]$
and hence that these two terms have same image under $\theta_{n,c}(f)$.
Thus:
\begin{eqnarray*}
g (\theta_{n,c}(f))(\alpha \otimes x)
&=& \alpha g(\theta_{n , c}(f)(g ^{-1}(1 \otimes x)))
= \alpha g (\theta_{n, c}(f) (1 \otimes x)) \\
&=& \alpha g (\varphi (f(x)))
= \alpha \varphi(g (f(x)))
= \theta_{n,c}(g(f))(\alpha \otimes x)
\end{eqnarray*}
and equivariance is proved.
\end{proof}

Recall that we have fixed an integer $N$ such that $u^N = 0$ in
$W[u]/E(u)^r$ and defined $b = \frac N {p-1}$ and $a = \frac {pN}
{p-1}$. Combining Propositions \ref{prop:prba} and \ref{forcompatible},
we directly get the following.

\begin{co}
\label{n&l2}
The morphism $\hat \rho_{\infty,b} : \hat T_{\gs_n}(\M) \to \hat
J_{n,b}(\M)$ is injective and its image is $\hat \rho_{a,b }(\hat
J_{n,a}(\M))$.
\end{co}

\begin{theorem}
\label{compatible}
With previous notations, the map $\theta_n : T_{\gs_n}(\M) \ito \hat T_n
(\hM) \simeq T $ is an isomorphism of $\Z_p[G_s]$-modules for all
integer $s > \smin = n-1+\log_p(\frac N e)$.
\end{theorem}

\begin{proof}
We already know that $\theta_n$ is bijective (Theorem \ref{main}.(1)).
Now, consider the following commutative diagram:
$$\xymatrix @C=60pt @R=15pt {
T_{\gs_n}(\M) \ar[r]^{\theta_n}_{\sim} \ar[d]_{\rho_{\infty, a}}
& \hat T_n (\hM) \ar@{^(->}[dd]^{\hat \rho_{\infty, b}} \\
J_{n,a} (\M) \ar[d]_{\rho_{a, b}} \\
J_{n,b} (\M) \ar[r]^{\theta_{n, b}}_{\sim} & \hat J_ {n,b} (\hM)}$$
Note that $\smin = s_1(a) = s_2(b)$. Thus by definition $G_s$-action on
$T_{\gs_n}(\M)$ (resp. by Proposition \ref{prop:prolong}, resp. by
Proposition \ref{forcompatible}), $\rho_{\infty,a}$ (resp. $\rho_{a,b}$,
resp. $\theta_{n,b}$) is $G_s$-equivariant. Since $\hat \rho_{\infty,b}$
is injective (Corollary \ref{n&l2}) and $G$-equivariant, we deduce that
$\theta_n$ is also $G_s$-equivariant as claimed.
\end{proof}

We end this section by giving a proof of Theorem \ref{intro:pleinfid} of
introduction. For convenience of the reader, we first recall its
statement:

\begin{co}[Theorem \ref{intro:pleinfid}]
\label{cor:pleinfid}
Let $V$ and $V'$ be two semi-stable representations of $G$. Let $T$
(resp. $T'$) a quotient of two $G$-lattices in $V$ (resp. $V'$) which is
killed by $p^n$. Then any morphism $G_\infty$-equivariant $f : T \to T'$
is $G_s$-equivariant for all integer $s > n - 1 + \log_p(nr)$.
\end{co}

\begin{proof}
Consider $\M$ (resp. $\M'$) some Kisin module such that $T_{\gs_n}(\M) =
T$ (resp. $T_{\gs_n}(\M') = T'$). We may assume that $\M$ and $\M'$ are
maximal in the sense of \cite{carliu}. Then by Corollary 3.3.10 of
\emph{loc. cit.}, $f$ comes from a morphism $g : \M' \to \M$.  Using
Theorem \ref{compatible}, one easily see that $T_{\gs_n}(g) = f$ is
$G_s$-equivariant.
\end{proof}

\section{Ramification bound}

In this section, we give proofs of Theorems \ref{intro:ramgs} and
\ref{intro:main} based on above preparations. Our strategy is similar to
those in \cite{intro:abrashkin}, \cite{intro:fontaine},
\cite{intro:hattori} and \cite{intro:hattori2}.
Let $n$ be a positive integer. Recall that we have defined several
constants, that are:
\begin{itemize}
\item $N$ is an integer such that $u^N = 0$ in $W_n[u]/E(u)^r$ (recall
also that one may choose $N = ern$);
\item $b = \frac N {p-1}$ and $a = \frac {pN} {p-1}$;
\item $s_0(a) = n - 1 + \log_p(\frac a e) = n + \log_p(\frac N {e(p-1)})$;
\item $\smin = s_1(a) = s_2(b) = n - 1 + \log_p(\frac N e)$.
\end{itemize}

Note that if we have chosen $N = ern$, then $s_0(a)$ is nothing but the
minority of $s$ that appear in Theorem \ref{intro:ramgs}. Let $T = L/L'$
be a quotient of two lattices in a semi-stable representation and assume
that $T$ is killed by $p^n$. Since we have a surjective map $L/p^n L \to
T$, it is enough to bound ramification for $L/p^n L$. Hence, without
loss of generality, we will assume that $T$ is free of $\Z/p^n \Z$. By
Theorem \ref{main}, there exists a $(\varphi, \hat G)$-module $\hM$ such
that $\hat T_n(\hM) = T$. With our extra assumption, $\M$ is finite free
over $\gs_n$.

>From now on, we fix an integer $s > s_0(a)$. Remark that $s_0(a) >
\smin$ so that we also have $s > \smin$. Hence theory developed in
previous sections applies. In particular, by Propositions
\ref{prop:wnrz} and \ref{prop:prolong}, for all $c \in [0, e
p^{s-n+1}[$, we have a $G_s$-equivariant isomorphism
\begin{equation}
\label{eq:jncjnsc}
J_{n,c}(\M) \simeq
J^{(s)}_{n,c}(\M): =\t{Hom}_{\gs, \varphi } \left(\M , W_n(k)
\otimes_{W_n(k), \sigma^s} \frac{W_n(\O_{\bar K}/p)}
{[\a_{\bar K}^{>c/p^s}]}\right)
\end{equation}
where the structure of $\gs$-module on $W_n(\O_{\bar K}/p)$ is given by
$u \mapsto 1 \otimes \pi_s$. Moreover, by Corollary \ref{n&l2} and
Theorem \ref{compatible}
\begin{equation}
\label{eq:Timab}
T|_{G_s} \simeq \t{im}\: \rho_{a,b} : J_{n,a}(\M) \to J_{n,b}(\M).
\end{equation}
Define $L$ to be the splitting field of $T$, that is, $L = (\bar K)^{\t{Ker}(\rho)}$, where  $\rho: G_K \to \GL_{\Z_p} (T)$ the attached group homomorphism. Set $L_s =K_s L$.

\subsection{The sets $J^{(s),E}_{n,c}(\M)$}

Let $E$ be an algebraic extension of $K_s$ inside $\bar K$. By
restriction, the valuation $v_K$ induces a valuation on $E$ and one may
define, for all non negative real number $c$, $\a_E^{\geqslant c} = \{x
\in E \,/\, v_K(x) \geq c\}$ and $\a_E^{>c} = \{x \in E \,/\, v_K(x) >
c\}$. If $c$ belongs to the interval $[0, e p^{s-n+1}[$, we put
$$J^{(s),E}_{n,c}(\M):
=\t{Hom}_{\gs, \varphi } \left(\M , W_n(k) \otimes_{W_n(k), \sigma^s}
\frac{W_n(\O_E/p)}{[\a_E^{>c/p^s}]}\right).$$
They are $\Z_p[G_s]$-modules. As usual, if $0 \leq c \leq c' \leq e
p^{s-n+1}$, we have a natural $G_s$-equivariant morphism
$\rho^{(s),E}_{c',c} : J^{(s),E}_{n,c'}(\M) \to J^{(s),E}_{n,c}(\M)$.
Apparently $J^{(s),E}_{n,c}(\M)$ injects $J^{(s),\bar K}_{n,c}(\M) =
J^{(s)}_{n,c}(\M)$.

The aim of this subsection is to show the following theorem.

\begin{theorem}
\label{splitting}
Notations as above. The natural injection $\rho^{(s),E}_{a,b}(J^{(s),E}
_{n,a}(\M)) \subset \rho_{a,b}(J_{n,a}(\M))$ is bijective if and only if
$L_s \subset E$.
\end{theorem}

\begin{remark}
By (\ref{eq:Timab}), $\rho_{a,b}(J_{n,a}(\M))$ is canonically isomorphic
to $T$ as a $\Z_p[G_s]$-module.
\end{remark}

In order to achieve the proof, we will need to lift
$J^{(s),E}_{n,c}(\M)$ at $\O_E$-level. We begin by defining a map
$\varphi: W_n(\O_E) \to W_n (\O_E)$ by
$$\varphi(z):= (z_0 ^p, \dots, z_{n-1}^p).$$
Note that $\varphi$ is \emph{not} a ring homomorphism. Nevertheless one
easily check that $\varphi([\lambda]z)= [\lambda^p] \varphi(z)$ for
$\lambda \in \O_E $ and $z \in W_n(\O_E)$ and $\varphi$ is
$G_s$-equivariant.

\begin{remark}
If $A$ is any ring, one can always define Frobenius $\phi: W(A) \to
W(A)$ by $w_m(\phi(x))= w_{m+1}(x), \, \forall x \in W(A)$, where
$w_n(x)$ is the $m$-th ghost component of $x$. Then $\phi$ can be proved
to be a ring homomorphism (see p.14 in \cite{Hazewinkel}).
Unfortunately, such Frobenius does not preserve the kernel of natural
projection $W(A) \to W_n(A)$ unless $A$ has characteristic $p$. Hence
it is not well-defined on $W_n (A)$ if $A$ has characteristic $0$.
\end{remark}

Recall now that we have assumed that $\M$ is finite $\gs_n$-free.
Select a basis $(e_1, \dots, e_d)$ of $\M$ and write $\varphi(e_1,
\dots, e_d)= (e_1, \dots, e_d) A$ with $A\in \t{M}_d(\gs_n)$. As
discussed in \S2.3, there exists $B\in \t{M}_d (\gs_n)$ such that $AB=
u^{N}I $. Let $\tilde A$ and $\tilde B$ be matrices in $\t{M}_d(W_n
(\O_{K_s}))$ that respectively lifts images of $A$ and $B$ under the
ring homomorphism $\gs_n \to W_n(\O_{K_s}/p)$, $u \mapsto \pi_s$,
$\lambda \mapsto \sigma^{-s}(\lambda)$ ($\lambda \in W_n(k)$).
Apparently, $\tilde A\tilde B \equiv [\pi_s]^N \pmod{W_n(p \O_{K_s})}$.
Hence, using condition on $s$, one prove that there exists a matrix $R$
with coefficients in $W_n([\a_{K_s}^{>0}])$ such that $\tilde A \tilde B
= [\pi_s]^N (I + R)$ (where $I$ is the identity matrix). Noting that
$I+R$ is invertible, one get $\tilde A \tilde B (I+R)^{-1} = [\pi_s]^N
I$. Hence, up to replacing $\tilde B$ by $\tilde B (I+R)^{-1}$, one may
assume that $\tilde A \tilde B = [\pi_s]^N I$.
Finally define a $G_s$-set
$$\tilde J^{(s),E}_{n}(\M):
= \big\{(\tilde x_1, \dots, \tilde x_d) \in W_n(\O_E)^d \, / \,
(\varphi(\tilde x_1), \dots, \varphi(\tilde x_d)) = (\tilde x_1, \dots,
\tilde x_d) \tilde A \big\}.$$
The natural projection $W_n (\O_E) \to W_n (\O_E/ p) \to W_n (\O_E/p)
/[\a_E^{>c/p^s}]$ induces a map $\tilde \rho^{(s),E}_{c}:
\tilde J^{(s),E}_{n}(\M) \to J^{(s),E}_{n, c}(\M)$.

\begin{lemma}
\label{lift}
$\tilde \rho^{(s),E}_{b}$ is injective and its image is
$\rho^{(s),E}_{a,b} (J^{(s),E}_{n,a} (\M))$.
\end{lemma}

\begin{proof}
During the proof, if $z$ is any element in $W_n(\O_E)$, we will denote
by $z^{(i)} \in \O_E$ its $i$-th component. By the same way, we define $Z^{(i)}$
for a matrix $Z$ with entries in $W_n(\O_E)$. Also, if $Z$ is a matrix
with entries in $\O_E$, we will denote by $v_K(Z)$ the smallest
valuation of coefficients of $Z$.

We first show $\tilde \rho^{(s),E}_{b}$ is an injection. Assume that
$X$ and $Y$ are in $\tilde J^{(s),E}_{n}(\M)$ such that
$\tilde \rho^{(s),E}_{ b}(X) - \tilde \rho^{(s),E}_ {b}(Y)=0$. Then $Z =
X - Y \in [\a_E^{>b/p^s}] + W_n(p \O_E) = [\a_E^{>b/p^s}]$. We need to
prove that $Z = 0$. Assume by contradiction that is it false and
consider $m$ the smallest number such that $Z^{(m)} \neq 0$. Define $W := Z
\tilde A = \varphi(X)- \varphi (Y) = \varphi(Y + Z)- \varphi(Y)$. Easy
computations show that $W^{(i)} = 0$ for $i < m$ and
\begin{equation}
\label{eq:Wm}
W^{(m)} = \sum\limits_{i=1}^{p}{p \choose i} (Y^{(m)})^{p-i} (Z^{(m)})^i.
\end{equation}
where the multiplication is computed component by component. If $1 \leq i
< p$, we have
$$v_K \left({p \choose i} (Y^{(m)})^{p-i} (Z^{(m)})^i \right)
\geq e + v_K(Z^{(m)}) > N p^{m-1-s} + v_K(Z^{(m)})$$
and, using $v_K(Z^{(m)}) > b p^{m-1-s}$ (recall that $Z \in
[\a_E^{>b/p^s}]$), we find
$$v_K((Z^{(m)})^p) > (p-1) b p^{m-1-s} + v_K(Z^{(m)}) =
N p^{m-1-s} + v_K(Z^{(m)}).$$
Hence each term in RHS of (\ref{eq:Wm}) has valuation greater than
$N p^{m-1-s} + v_K (Z^{(m)})$. So $v_K(W^{(m)}) > N p^{m-1-s} + v_K
(Z^{(m)})$. But, on the other hand, comparing the $m$-th component of
$W \tilde B = [\pi_s]^N Z$, we get $v_K (W^{(m)}) \leq N p^{m-1-s} + v_K
(Z^{(m)})$. This is a contradiction and injectivity follows.

Let us now prove the second statement. Remark first that for all $c \in
[0, e p^{s-n+1}[$, we have $W_n(\O_E/p)/[\a_E^{>c/p^s}] \simeq
W_n(\O_E)/[\a_E^{>c/p^s}]$ and hence that $J^{(s),E}_{n, c}(\M)$ can
be identified with
$$\big\{(\tilde x_1, \dots, \tilde x_d) \in W_n(\O_E)^d \, / \,
(\varphi(\tilde x_1) , \dots, \varphi(\tilde x_d)) \equiv (\tilde x_1,
\dots, \tilde x_d ) \tilde A \pmod {[\a_E^{>c/p^s}]}\big\}.$$

Let $X= (\tilde x_1, \dots, \tilde x_d) \in W_n(\O_E)^d $ be an solution
as above. We have equation $\varphi (X)= X \tilde A + Q'$ with
coefficients of $Q'$ in $[\a_E ^{> a/p^s}]$. Actually, the congruence
holds in $[\a_E^{\geqslant a'}]$ for some $a'$ satisfying $ \frac{e}{p
^{ n-1}}\geq a' > \frac a {p^s}$. Note that $\frac{e}{p^{n-1}} \geq a'$
implies that $W_n (\O_E) \subset [\a_E^{\geqslant a'}]$.
For the rest of the proof, fix $\alpha \in \O_E$ some element of
valuation $a'$. By the similar argument as in Lemma \ref{lem:aRc}.(2),
$[\a_E^{\geqslant a'}]$ is the principal ideal generated by $[\alpha]$.
Therefore, one have $\varphi(X) - X\tilde A = [\alpha] Q$ with the
coefficients of $Q$ in $W_n (\O_E)$. We want to prove that there exists
a matrix $Y$ with coefficients in $[\a_E^{>b/p^s}]$ such that $ (X +
Y)\tilde A = \varphi(X + Y)$. Let us search $Y$ of the shape $[\beta] Z$
with $\beta = \frac{\alpha}{\pi_s^N}$ (which belongs to $\O_E$ because
of valuations) and coefficients of $Z$ in $W_n(\O_E)$. Our condition
then becomes $(X + [\beta]Z)\tilde A = \varphi (X+ [\beta] Z)$.
Multiplying $\tilde B $ on both sides and noting that $[\pi_s]$ is a
non-zero divisor of $W_n (\O_E)$, we need to prove that the following
equation has a (necessarily unique) solution:
\begin{equation}\label{solution}
[\pi_s]^N X + [\pi_s ]^N [\beta] Z = \varphi (X+ [\beta] Z) \tilde B.
\end{equation}
Let us prove by induction on $n$. If $n=1$, set $Z_0 = 0$ and $Z_{l+1}=
\pi_s ^{-N}\beta^{-1}(\varphi(X + \beta Z_l)\tilde B- \pi_s^NX)$. To see
that $Z_{l+1}$ is in $\O_E$, note that
$$\begin{array}{l}
\varphi(X + \beta Z_l)\tilde B -\pi_s^NX = (X+ \beta Z_l)^p \tilde B -
\pi_s ^NX \\
\hspace{1.5cm} = ( \varphi (X)\tilde B - \pi_s ^NX) + \sum \limits_{i
=1}^{p-1} {p \choose i} X^{p-i} (\beta Z_l) ^{i}\tilde B + \beta ^p
(Z_l)^p \tilde B.
\end{array}$$
Since $\varphi (X)\tilde B - \pi_s ^NX = \alpha Q \tilde B $, $v_K (\pi_s^N \beta )\leq v_K (\alpha) \leq v_K (p) $ and $(p-1)v_K(\beta) \geq v_K (\pi_s^N) $, we see that
$Z_{l+1}$ is in $\O_E$. Note that
\begin{eqnarray*} Z_{l+1}- Z_l &=& \pi_s^{-N} \beta^{-1} (\varphi(X + \beta Z_l)- \varphi (X + \beta Z_{l-1}) )\tilde B\\
&=& \pi_s^{-N} \beta^{-1} \sum _{i=1}^p {p \choose i} X^{p-i} \beta^{i}(Z_l ^{i}- Z_{l-1}^i)\tilde B .
\end{eqnarray*}
Since $v_K(p) \geq v_K(\pi_s^N \beta)$ and $ (p-1) v_K(\beta)>v_K
(\pi_s^N)$, we see that $v_K (Z_{l+1}- Z_l) \geq \gamma+ v_K (Z_{l}-
Z_{l-1})$, where $\gamma = \min (v_K(\beta), v_K
(\beta^{p-1}\pi_s^{-N}))>0 $. Hence $Z_l$ converge to $Z$ and we solve
the equation \eqref{solution} for $n=1$.

Now assume that equation \eqref{solution} has a solution for $n \leq
m-1$, consider the $n =m$ case. Recall that $z^{(i)} \in \O_E$
represents the $i$-th component of $z\in W_m (\O_E)$. Set $Z_0 =
(Z^{(0)}_0, \dots, Z^{(m-1)}_0)$ where $Z_0^{(m-1)}=0$ and $(Z_0^{(0)},
\dots , Z^{(m-2)}_0)$ is the solution of \eqref{solution} in $n=m-1$
case. Now set $Z_{l+1}= [\pi_s] ^{-N} [\beta] ^{-1}(\varphi(X + [\beta]
Z_l)\tilde B- [\pi_s]^NX)$. Since $(Z_0^{(0)}, \dots , Z^{(m-2)}_0)$ is
the solution of \eqref{solution} in $n=m-1$ case, we see that $Z^{(i)}_l
= Z_{l+1}^{(i)}$ for all $l$ and $i=0, \dots, m-2$.  Now it suffices to
check that $Z_{l+1}$ has coefficients in $W_m (\O_E)$ and $Z_l$
converges.

Since $\varphi(X+ [\beta] Z_l)= \varphi (X)+ \varphi ([\beta] Z_l) $ in
$W_m (\O_E/ p)$, we have $\varphi(X+ [\beta] Z_l) = \varphi (X) +
\varphi ([\beta] Z_l) + C'$ with coefficients of $C'$ in $W_m (p\O_E)$.
Since $W_m (p\O_E) \subset [\alpha] W_m (\O_E)$, we can write $C' =
[\alpha] C$ with coefficients of $C$ in $W_m (\O_E)$.
Hence $\varphi(X + [\beta]Z_l)\tilde B- [\pi_s]^NX= (\varphi(X) \tilde B
- [\pi_s]^N X)+ [\beta]^{p} \varphi (Z_l)\tilde B + [\alpha]C \tilde B =
[\alpha] Q\tilde B+ [\beta]^{p} \varphi (Z_l)\tilde B + [\alpha] C
\tilde B$. Since $(p-1)v_K(\beta) > v_K (\pi_s ^N)$, we see that
$Z_{l+1}$ is well defined. Now $[\pi_s]^{N} [\beta] (Z_{l+1}- Z_l)=
(\varphi(X + [\beta] Z_l)- \varphi (X + [\beta] Z_{l-1}) )\tilde B =W
\tilde B $, where
\begin{eqnarray*}
  W&=& \varphi(X + [\beta] Z_l)- \varphi (X + [\beta] Z_{l-1})\\
  &=& \varphi(X + [\beta]  Z_{l-1} + [\beta] (Z_l-Z_{l-1}))-\varphi (X +
  [\beta]  Z_{l-1})\\&=& \varphi (V + [\beta] (Z_l-Z_{l-1}))-\varphi
  (V)
\end{eqnarray*}
with $V = X + [\beta] Z_{l-1}$. Since $Z^{(i)}_l = Z_{l+1}^{(i)}$ for
all $l$ and $0 \leq i \leq m-2$, $W^{(i)}= 0$ for $i =0, \dots, m-2$ and
$$W^{(m-1)} = \sum_{i=1}^{p} {p \choose i} (V^{(m-1)})^{p-i} (\beta
^{p^{m-1}} (Z^{(m-1)}_l - Z^{(m-1)}_{l-1 }))^{i}.$$

Hence $v_K ((\pi_s^{N} \beta)^{p^{m-1}}) + v_K (Z^{(m-1)}_{l+1} -
Z^{(m-1)}_{l }) \geq v_K (\beta ^{p^{m}}) +v_K (Z^{(m-1)}_l -
Z^{(m-1)}_{l-1 }) $ if $i=p$ and
$$v_K ((\pi_s^{N}\beta)^{p^{m-1}}) + v_K (Z^{(m-1)}_{l+1} - Z^{(m-1)}_{l
})) \geq v_K (p) + v_K(\beta) + v_K (Z^{(m-1)}_l - Z^{(m-1)}_{l-1 })$$
if $1 \leq i \leq p-1$. Since $(p-1)v_K(\beta) > v_K(\pi_s ^N)$ and $v_K
((\pi_s^N\beta)^{ p^{m-1}}) \leq v_K (p)$, we get $$v_K (Z^{(m-1)}_{l+1}
- Z^{(m-1)}_{l })) \geq \gamma + v_K (Z^{(m-1)}_l - Z^{(m-1)}_{l-1 }),$$
where $\gamma = \min (v_K (\beta ), v_K(\beta^{p-1}\pi_s^{-N}))$. Hence
$Z_l$ converges and we are done.
\end{proof}

\begin{proof}[Proof of Theorem \ref{splitting}]
We have $G_s$-equivariant bijections of sets:
$$\begin{array}{rcll}
\tilde J^{(s), \bar K }_{n} (\M) & \simeq &
\rho^{(s)}_{a,b} (J^{(s)}_{n, b}(\M)) & \t{by Lemma \ref{lift} applied
with }E=\bar K\\
& \simeq & \rho_{a,b} (J_{n, b}(\M)) & \t{by Formula (\ref{eq:jncjnsc}}) \\
& \simeq & T|_{G_s} & \t{by Proposition \ref{prop:prba} and Theorem \ref{compatible}.}
\end{array}$$
Taking fixed points under $\gal(\bar K/E)$, we get a natural bijection
$\tilde J^{(s),E}_{n} (\M) \simeq T^{\gal(\bar K/E)}$. Hence again by
Lemma \ref{lift}, $T^{\gal(\bar K/E)} \simeq \rho^{(s),E}_{a,b}
(J^{(s),E}_{n,a} (\M))$, from what the theorem is easily deduced.
\end{proof}

\subsection{Proof of Theorem \ref{intro:ramgs}}
\label{subsec:ramgs}

Recall that $L_s=K_s L$ with $L$ the splitting field of $T$. Now we are
ready to bound the ramification of $L$.  To do this, we need to recall
the property $(P_m^{F/N})$ described by Fontaine (Proposition 1.5,
\cite{intro:fontaine}). First, in order to fix notations, we would like
to recall some definitions about ramification filtration, although we
refer to \emph{loc. cit.} and \cite{serre}, Chap. IV. for basic
properties.

Let $F/N$ be a Galois extension of $p$-adic fields, with Galois group
$G$. For all non negative real number $\lambda$, we define a normal
subgroup $G_{(\lambda)}$ of $G$ by
$$G_{(\lambda)} = \{ \sigma \in G \, / \, v_N(\sigma(x) - x) \geq
\lambda, \, \forall x \in \O_N\}$$
where $v_N$ is the valuation normalized by $v_N(N^\star) = \Z$ and
$\O_N$ is the ring of integers of $N$. We underline that we use here
conventions of \cite{intro:fontaine} and that they differ by a shift
with conventions of \cite{serre}, Chap. IV. By definition $G_{(\lambda)}$
is called the \emph{lower ramification filtration} of $G$. Now, let
$\varphi_{N/K} : [0,+\infty[ \to [0, +\infty[$ be the function defined
by
$$\varphi_{N/K}(\lambda) := \int_0^\lambda \frac{\card G_{(t)}} {\card
G_{(1)}} dt.$$
It is increasing, continuous, concave, piecewise affine and bijective.
Let $\psi_{N/K}$ denote its inverse and set $G^{(\mu)} =
G_{(\psi_{N/K}(\mu))}$: it is the \emph{upper ramification filtration}.
Finally call $\lambda_{F/N}$ (resp. $\mu_{F/N}$) the last break in the
lower (resp. upper) ramification filtration of $G$, that is the infimum
of $\lambda$ (resp. $\mu$) such that $G_{(\lambda)} = 1$ (resp.
$G^{(\mu)} = 1$). Obviously $\mu_{F/N} = \varphi_{F/N}(\lambda_{F/N})$.

\begin{prop}[Fontaine]
\label{fontaine}
Let $F$ and $N$ be finite extensions of $K$ with $N \subset F \subset
\bar K$ and $F$ is Galois. For any positive real number $m$, consider the following
property
$$(P^{F/N}_m)
\begin{cases} \text{For any algebraic extension } E \t{ over } N. \text{
If there exists an } \O_N \text{-algebra} \\
\text{homomorphism } \O_F \to \O_E/\a^{>m}_E, \text{ then there
exists a }N\text{-injection } F \inj E.
\end{cases}$$
If $(P^{F/N}_m)$ is true, then $\mu_{F/N} \leq e_{N/K} m + \frac
1{e_{F/N}}$ where $e_{N/K}$ (resp. $e_{F/N}$) is the ramification index
of $N/K$ (resp. of $F/N$).
\end{prop}

We will also need the following corollary:

\begin{co}
\label{different}
If $(P^{F/N}_m)$ holds for a positive real number $m$, then $v_K
(\D_{F/N}) < m$.
\end{co}

\begin{proof}
If $F/N$ is unramified, $v_K (\D_{F/N}) = 0$ and the corollary is
obvious. If not, by Proposition 1.3 of \cite{intro:fontaine}, we see
that $e_{N/K} v_K(\D_{F/N}) = \mu_{F/N} - \lambda_{F/N}$. If $(P^{F/N}_m)$
holds then $v_K(\D_{F/N}) \leq m - e^{-1}_{N/K}(\lambda_{F/N} -
e_{F/N}^{-1})$. Conclusion then follows from that $e_{F/N} > 1$ and
$\lambda_{F/N} \geq 1$ (both are true because $F/N$ is assumed to
be ramified).
\end{proof}

We claim that $(P_m^{L_s/K_s})$ holds for $m = a p^{n-1-s}$. To see
this, pick $f : \O_{L_s} \to \O_E/\a^{>m}_E$ an $\O_{K_s}$-algebra
homomorphism. Obviously, for any real number $c \in [0, m]$, $f$ induces
a map $f_c : \O_{L_s}/\a^{>c}_{L_s} \to \O_E/\a^{>c}_E$.

\begin{lemma}
\label{connect2}
\begin{enumerate}
\item For any $c \leq m$, $f_c$ is injective.
\item For any $c \leq a$, $f_{c p^{n-1-s}}$ induces an injection
$$W_n(\O_{L_s}/ p)/ [\a_{L_s}^{>c/p^s}] \inj W_n(\O_E/p)/[\a_E^{>c/p^s}].$$
\end{enumerate}
\end{lemma}

\begin{proof}
(1) It is the same as the proof of Lemma 4.4 of \cite{intro:hattori}.

(2) Using an analogue of Lemma \ref{lem:aRc}.(1), one easily prove that
natural projections $\O_{L_s}/p \to \O_{L_s}/\a_{L_s}^{>c p^{n-1-s}}$
and $\O_E/p \to \O_E/\a_E^{>c p^{n-1-s}}$ induce isomorphisms
\begin{eqnarray*}
W_n(\O_{L_s}/ p)/[\a_{L_s}^{>c/p^s}] & \simeq &
W_n(\O_{L_s}/ \a_{L_s}^{>c p^{n-1-s}})/ [\a_{L_s}^{>c/p^s}] \\
W_n(\O_E/ p)/[\a_E^{>c/p^s}] & \simeq &
W_n(\O_E/ \a_E^{>c p^{n-1-s}})/ [\a_E^{>c/p^s}].
\end{eqnarray*}
Hence $f_{c p^{n-1-s}}$ indeed induces a map $$W_n(\O_{L_s}/\a_{L_s}^{>c p^{n-1-s}})/
[\a_{L_s}^{>c/p^s}]  \to W_n(\O_E/ \a_E^{>c p^{n-1-s}})/ [\a_E^{>c/p^s}]$$ and checking injectivity is
now straitforward using (1).
\end{proof}

\noindent
Thus, we get injections:
$$\rho_{b,a}^{(s),L_s}(J_{n,a}^{(s),L_s}(\M)) \inj
\rho_{b,a}^{(s),E}(J_{n,a}^{(s),E}(\M)) \inj
\rho_{b,a}^{(s)}(J_{n,a}^{(s)}(\M)) = \rho_{b,a}(J_{n,a}(\M) \simeq T$$
the first one being induced by $f$ (which is obviously compatible with
Frobenius since it is a ring homomorphism). By Theorem \ref{splitting},
LHS is isomorphic to $T$. The composite map is then an injective
endomorphism of $T$. Consequently, it is an isomorphism because $T$ is
finite. It follows that $\rho_{b,a}^{(s),E} (J_{n,a}^{(s),E}(\M)) \inj
\rho_{b,a}^{(s)}(J_{n,a}^{(s)}(\M))$ is bijective and then, applying
again Theorem \ref{splitting}, we get $L_s \subset E$. Property
$(P_m^{L_s/K_s})$ is proved.

\begin{co}
\label{co:ramgs}
For all integer $s > s_0(a) = n + \log_p(\frac{N} {e(p-1)})$ and all
real number $\mu > \frac{N p^n} {p-1}$, $G_s^{(\mu)}$ acts trivially
on $T$.
\end{co}

\begin{proof}
If $L_s/K_s$ is tamely ramified, then $\mu_{L_s/K_s} \leq 1$ and we
are done (note that $\frac{N p^n} {p-1} \geq 1$). Otherwise, by
Proposition \ref{fontaine}, we obtain:
$$\mu_{L_s/K_s} \leq e_{K_s/K} m + \frac 1{e_{L_s/K_s}} = \frac{N p^n}
{p-1} + \frac 1{e_{L_s/K_s}}.$$
Since $L_s/K_s$ is wildly ramified, we know that $e_{L_s/K_s}
\mu_{L_s/K_s}$ lies in $p \Z$. From $e_{L_s/K_s} \mu_{L_s/K_s} (p-1)
\leq e_{L_s/K_s} N p^n + (p-1)$ we then deduce $e_{L_s/K_s}
\mu_{L_s/K_s} (p-1) \leq e_{L_s/K_s} N p^n$, \emph{i.e.} $\mu_{L_s/K_s}
\leq \frac{N p^n}{p-1}$.
\end{proof}

Now taking $N = ern$, we get Theorem \ref{intro:ramgs}. (Recall that
$ern$ is not in general the best value one can choose (expect for
$n=1$). See \S \ref{subsec:sharpness} for a discussion about this.)

As a conclusion, we would like to emphasize that previous Corollary is
valid for all choice of $\pi$ and compatible system $(\pi_s)$ of
$p^s$-roots of $\pi$. Therefore
$$L \subset \bigcap_{s,\mu,K_s} K_s^{(\mu)}$$
where $K_s^{(\mu)}$ is the extension of $K_s$ cut by $G_s^{(\mu)}$
and where the intersection runs over all $(s,\mu)$ satisfying conditions
of the Corollary and all extensions $K_s$ over $K$ generated by a
$p^s$-root of some uniformizer of $K$. It should maybe be interesting to
understand better this intersection.

\subsection{Proof of Theorem \ref{intro:main}}

Consider $\alpha$ and $\beta$ such that $\frac N {e(p-1)} = p^\alpha
\beta$ with $\alpha \in \mathbb N$ and $\frac 1 p < \beta \leq 1$. (If
$N = ern$, then $\alpha$ and $\beta$ are those of Theorem
\ref{intro:main}). From now on, we fix $s = n + \alpha$. One certainly
have that $s \geq s_0(a) = n + \log_p(\frac N {e(p-1)})$ as it was
assumed at the beginning of this section.

This is very easy now to bound valuation of $\D_{L/K}$. We just write:
\begin{eqnarray*}
v_K(\D_{L_s/K}) & = & 1 + es - \frac 1{p^s} +
v_K(\D_{L_s/K_s}) < 1 + es - \frac 1{p^s} + ap^{n-1-s} \\
& = & 1 + es - \frac 1{p^s} + e p^{\alpha+n-s} \beta =
1 + e(n + \alpha + \beta) - \frac 1{p^{n + \alpha}}
\end{eqnarray*}
where the inequality $v_K(\D_{L_s/K}) < ap^{n-1-s}$ follows from
Corollary \ref{different} and the fact that $(P_{ap^{n-1-s}}^{L_s/K_s})$
holds as it was seen before. Now, since $L$ is a subextension of $L_s$,
we have $v_K(\D_{L/K}) \leq v_K(\D_{L_s/K})$ and the previous bound
works also for $v_K(\D_{L/K})$. Taking $N = ern$, we get Theorem
\ref{intro:main}.(2).

To bound $u_{L/K}$, we first need a kind of transitivity formula:

\begin{lemma}
\label{lem:swancompo}
Let $N \subset F$ two finite Galois extensions of $K$. Then
$\mu_{F/K} = \max (\mu_{N/K}, \varphi_{N/K}(\mu_{F/N}))$.
\end{lemma}

\begin{proof}
Let $G$ (resp. $H$) denote the Galois group of $F/K$ (resp. $F/N$).
Since $N$ is Galois over $K$, $H$ is a normal subgroup in $G$. Using
definition, one directly check $H_{(\lambda)} = G_{(\lambda)} \cap H$.
Using $\varphi_{F/K} = \varphi_{N/K} \circ \varphi_{F/N}$ (see
\cite{serre}, Chap. IV, Proposition 15), and taking $\lambda =
\psi_{F/K} (\mu)$, one obtain $H^{(\psi_{N/K}(\mu))} = G^{(\mu)} \cap H$.
On the other hand, Proposition 14 of \emph{loc. cit.} gives
$(G/H)^{(\mu)} = G^{(\mu)} H / H$. Combining both results, we see that
the sequence
$$0 \to H^{(\psi_{N/K}(\mu))} \to G^{(\mu)} \to (G/H)^{(\mu)} \to 0$$
is well defined and exact for all $\mu$. Conclusion is then easy.
\end{proof}

Let $N_s := K_s(\zeta_{p^s})$ and $F_s := L N_s$ where $\zeta_{p^s}$ is
a $p^s$-th root of unity. Note that $N_s$ and $F_s$ are Galois
extensions. We would like to apply Lemma \ref{lem:swancompo} with $N =
N_s$ and $F = F_s$. First we claim $(P_{a p^{n-1-s}} ^{F_s/N_s})$ holds.
Indeed, given an $\O_{N_s}$-algebra homomorphism $f : \O_{F_s} \to \O_E/
\a^{> m}_E$ (with $m = a p^{n-1-s}$), the restriction $f|_{\O_{L_s}} :
\O_{L_s} \to \O_E/ \a^{>m}_E$ is an $\O_{K_s}$-algebra homomorphism.
Then, by the same argument below Lemma \ref{connect2}, we have a
$K_s$-injection $L_s \to E $. On the on the hand, since $N_s\subset E$
and $L$ are Galois, we see that there exists an $N_s$-injection $F_s=
L_sN_s \to E$. So $(P_{a p^{n-1-s}} ^{F_s/N_s})$ holds.  By the same
argument as in proof of Corollary \ref{co:ramgs}, we show
$$\mu_{F_s/N_s} \leq e_{N_s/K} a p^{n-1-s} = e_{N_s/K} \: \frac
N{p^{\alpha}(p-1)} = e_{N_s/K} \: e \beta$$
where $e_{N_s/K}$ is the ramification index of $N_s/K$. Furthermore, by
Remark 5.5 of \cite{intro:hattori2}, we already know that $\mu_{N_s/K} =
1 + e(s+\frac 1 {p-1})$. Hence, it remains to give an estimation of
$\varphi_{N_s/K}$. For this, note that there exists $\sigma \in
\gal(N_s/K)$ such that $\sigma \pi_s = \zeta_p \pi_s$ where $\zeta_p$ is
a primitive $p$-th root of unity. Hence $v_K(\sigma \pi_s - \pi_s) =
\frac e{p-1} + \frac 1{p^s}$ which implies
$$\lambda_{N_s/K} \geq e_{N_s/K} \left(\frac e{p-1} + \frac 1{p^s}
\right).$$
Moreover, using the definition, we see that $\varphi_{N_s/K}$ is affine
with slope $\frac 1 {e_{N_s/K}}$ after $\lambda_{N_s/K}$. Thus, by
concavity, one get
\begin{eqnarray*}
\varphi_{N_s/K}(\lambda)
& \leq & 1 + e\left(s + \frac 1 {p-1}\right) +
\frac 1{e_{N_s/K}} ( \lambda - \lambda_{N_s/K})  \\
& \leq & 1 + es + \frac {\lambda}{e_{N_s/K}} - \frac 1 {p^s}.
\end{eqnarray*}
Now apply Lemma \ref{lem:swancompo}:
$$\mu_{L/K} \leq \mu_{F_s/K} \leq 1 + e\left(s + \max \left(\beta,
\frac 1 {p-1}\right)\right).$$
Since $s = n + \alpha$, Theorem \ref{intro:main}.(1) is shown.

\section{Some results and questions about lifts}
\label{sec:lift}

\newcommand{\calC}{\mathcal C}
\newcommand{\Repf}{\t{Rep}_{\Z_p}}
\newcommand{\Rept}{\t{Rep}_{\t{tor}}}
\newcommand{\Reppn}{\t{Rep}_{\Z/p^n \Z}}
\newcommand{\Repp}{\t{Rep}_{\Z/p\Z}}
\newcommand{\tr}{\t{tr}}
\setcounter{theorem}{0}
\renewcommand\thetheorem{\thesection.\arabic{theorem}}

In this last section, we discuss some ideas about possible converses for
Theorem \ref{intro:main}. Precisely, we wonder when a given torsion
representation of $G_K$ can be realized as a quotient of two lattices in
a semi-stable (or even crystalline) representation, eventually with
prescribed Hodge-Tate weights. Denote by $\Repf(G_K)$ (resp.
$\Rept(G_K)$, resp. $\Reppn(G_K)$) the category of all
$\Z_p$-representations of $G_K$ that are finitely generated and free
(resp. killed by a power of $p$, resp. killed by $p^n$) as a
$\Z_p$-module. For any full subcategory $\calC$ of $\Repf(G_K)$, one can
always raise the following question

\begin{question}
\label{question}
For any $T \in \Rept(G_K)$ (resp. $T \in \Reppn(G_K)$), does there exist
$L$ and $L'$ in $\calC$ such that $T \simeq L/L'$?
\end{question}

Obviously if $\calC$ is stable under subobject (which will in general be
true in interesting examples), it is enough to find $L$ together with a
surjective $G_K$-equivariant morphism $L \to T$. In the sequel, we will
call a \emph{lift} such a morphism $L \to T$. If $\calC$ is moreover
stable by direct sum, the problem can be further reduced as follows.

\begin{prop}
\label{prop:killedp}
Assume that $\calC$ is stable under subobjects and direct sums. Assume
also that any $T \in \Repp(G_K)$ admits a lift $L \in \calC$. Then the
answer to Question \ref{question} is ``yes''.
\end{prop}

\begin{proof}
We make an induction on $n$. The case $n=1$ is obvious. Now assume the
statement is valid for $m \leq n-1$. Let $T$ be a representation killed
by $p^n$. Then we have an exact sequence $0 \to T' \to T \to T''\to 0 $,
where $T ' = p ^{n-1} T$ and $T'' = T/T'$. Since $T''$ is killed by
$p^{n-1}$, by induction, there exists an $L\in \calC $ that lifts $T''$.
Denote the surjections $L \to T''$ and $T \to T''$ by $f$ and $g$
respectively. Set $M:= T \times_{T''} L= \{(x, y) \in T \times L| g(x)=
f (y) \}$. Then we have an exact sequence $0 \to T' \to M \to L \to 0$.
Since $L$ is free over $\Z_p$, the sequence is split as $\Z_p$-module.
In particular $pM \simeq pL \oplus pT' = pL$ is finite free over
$\Z_p$. Now we have exact sequence $0 \to pM \to M \to M'\to 0 $ with
$M' = M / pM$. Since $M/ pM$ is killed by $p$, there exists an $L' \in
\calC$ such that $L'$ lifts $M / pM$. Set $N := M \times _{M'} L'$. It
sits in the exact sequence $0 \to pM \to N \to L' \to 0$, and since $pM$
and $L'$ are both finite free, $N$ is also. Note that $N$ is a lift of
$M$ hence a lift of $T$. Now it remains to show that $N $ is in $\calC$.
To see this, note that $N := M \times _{M'} L' \subset M \times L'$.
Then $pN \subset (pM) \times {pL'}$. But $pM \simeq pL \in \calC$. Hence
$pN \subset {pL} \times pL'$ belongs to $\calC$.
\end{proof}

We also have a kind of descent property:

\begin{prop}
\label{prop:descent}
Assume that $\calC$ is stable under subobjects and that the answer of
question \ref{question} is ``yes''. Let $L/K$ a finite extension.

Then, for any $T \in \Rept(G_L)$ (resp. $T \in \Reppn(G_L)$), there
exist $L$ and $L'$ in $\calC$ such that $T \simeq L/L'$ as
$G_L$-representations.
\end{prop}

\begin{proof}
By a previous remark, it is enough to show that $T$ admits a lift in
$\calC$. Let $T_0 := \t{Ind}^{G_L}_{G_K} (T)$. By assumption, there
exists a lift $L_0 \to T_0$ with $L_0 \in \calC$. But as
$G_L$-representations, $T_0 \simeq T^{[K:L]}$. Composing the map
$L_0 \to T^{[K:L]}$ with the first projection, we get the desired
lift.
\end{proof}

Nevertheless, of course, the answer to Question \ref{question} is in
general negative. For instance, we have the following theorem that can
be seen as a consequence of ramification bounds obtained in this
paper.

\begin{theorem}
For any $r >0$, answer to Question \ref{question} is ``no'' if $\calC$
is the category of lattices in semi-stable representations with
Hodge-Tate weights in $\{0, \ldots, r\}$.
\end{theorem}

\begin{proof}
There are several ways to prove this theorem. Below, we give two
different methods.

The first one is based on results shown in this paper. Select a Galois
extension $F/K$ which has very large ramification and let $T$ be the
regular representation with $\Z/p^n\Z$-coefficients of $\gal(F/K)$. Then
the splitting field of $T$ is $F$, and Theorem \ref{intro:main} shows
that $T$ cannot in general be lifted a semi-stable representation with
Hodge-Tate weights in $\{0, \dots, r\}$.

The second proof we would like to give uses the main result of
\cite{sec1:liu} which states that a finite free $\Z_p$-representation $L$
of $G_K$ is a lattice in a crystalline (resp. semi-stable)
representation with Hodge-Tate weights $\{0, \ldots, r\}$ if and only if
$L/p^n L$ is a quotient of two such lattices. Therefore, starting from a
representation $L$ such that $L \otimes_{\Z_p} \Q_p$ is not semi-stable,
there must exist an integer $n$ such that $L/p^n L$ gives a
counter-example to Question \ref{question} (with the category $\calC$ of
the theorem).
\end{proof}

Unfortunately, the above proof does not help us to solve the following
more interesting question:

\begin{question}
\label{q2}
Has Question \ref{question} a positive answer when $\calC$ is the
category of all lattices in semi-stable representations?
\end{question}

In fact, to check the above question, it suffices to look at
representations killed by $p$ (by Proposition \ref{prop:killedp}) and we
may assume that $K = \Q_p$ (by Proposition \ref{prop:descent}). Here are
some partial results in favor of a positive answer to Question \ref{q2}.

\begin{prop}
Let $T$ be a torsion representation of $G_\infty$. Then $T$ is a
quotient of two representations arising from finite free Kisin modules.
\end{prop}

\begin{proof}
By a similar argument as in proof of Proposition \ref{prop:killedp}, we
may assume that $T$ is killed by $p$. Let $M$ be the \'etale
$\varphi$-module over $k((u))$ attached to $T$ (see for instance
\cite{sec1:fontaine2}, A. 3). Since any torsion Kisin
module can be written as a quotient of two free Kisin modules, it is
enough to show that $M$ admits a submodule $\M$ which is a Kisin module
for some integer $r$. Let $(e_1, \ldots, e_d)$ be a basis of $M$ and $A$
be the matrix with coefficients in $k((u))$ such that
$$(\varphi(e_1), \ldots, \varphi(e_d)) = (e_1, \ldots, e_d) A.$$
Since changing all $e_i$'s in $u e_i$ changes $A$ in $u^{p-1} A$, one
may assume that $A$ has coefficients in $k[[u]]$. Furthermore, the
\'etaleness of Frobenius on $M$ exactly means that $A$ is invertible in
$k((u))$. Hence $\det A$ does not vanish. Finally, we choose $r$ such
that $\det A$ divides $u^{er}$ (that is $r \geq \frac 1 e \t{val}_u(\det
A)$) and we are done.
\end{proof}

\begin{theorem}
\label{theo:tame}
Any tamely ramified $\F_p$-representation of $G_K$ can be written as a
quotient of two lattices in a crystalline representation with Hodge-Tate
weights between $0$ and $p-1$.
\end{theorem}

\begin{remark}
In particular, the answer to Question \ref{q2} is yes if $T$ is tamely
ramified and killed by $p$.
\end{remark}

Before giving the proof of Theorem \ref{theo:tame}, we need to achieve
some computations with Kisin modules in the ``unramified case'' ($e=1$).
So from now, we assume $e=1$ (that is $K = W[1/p]$) and we fix $\pi = -p$
to be the chosen uniformizer, so that $E(u) = u + p$. We denote by $I$
the inertia subgroup of $G_K$ and for all integer $d$, we recall the
definition of the $d$-th fundamental tame inertia character $\theta_d$:
$$\begin{array}{rcl}
\theta_d \, : \, I & \to & \mu_{p^d-1}(\bar K) \simeq
\F_{p^d}^\star \\
g & \mapsto & \frac{g(\pi^{1/(p^d-1)})}{\pi^{1/(p^d-1)}}
\end{array}$$
where $\F_{p^d}$ is the subfield of $\bar k$ with $p^d$ elements.
For all $i \in \Z/d\Z$, set $\theta_{d,i} = \theta_d^{p^i}$. Let $(n_i)
_{i \geq 0}$ be a sequence of non negative integers that is periodic of
period $d$. Attached to it, define a filtered $\varphi$-module $D$ and a
free Kisin module $\M$ by
\renewcommand{\arraystretch}{1.2}
$$\left\{\begin{array}{l}
D = K e_0 \oplus K e_1 \oplus \cdots \oplus K e_{d-1} \\
\varphi(e_{i+1}) = p^{n_i} e_i \quad (i \in \Z/d\Z) \\
\Fil^n D = \sum_{n_i \geq n} K e_{i+1} \quad (i \in \Z/d\Z)
\end{array}\right.
\qquad
\left\{\begin{array}{l}
\M = \gs \mathfrak e_0 \oplus \gs \mathfrak e_1 \oplus \cdots
\oplus \gs \mathfrak e_{d-1} \\
\varphi(\mathfrak e_{i+1}) = (u+p)^{n_i} \mathfrak e_i \quad (i \in
\Z/d\Z)
\end{array}\right.$$
If $\psi$ is a character of $I$ with values in $\F_{p^d}$, we will
denote by $\F_{p^d}(\psi)$ the one-dimensional representation of $I$
given by $\psi$. Our technical lemma is the following.

\begin{lemma}
\label{lem:computkisin}
We keep previous notations.
\begin{enumerate}
\item One have a canonical $G_\infty$-equivariant $\Q_p$-linear
isomorphism
$$T_\gs(\M) \otimes_{\Z_p} \Q_p \simeq \hom_{\varphi,\Fil}(D,
B^+_\cris).$$
\item Setting $I_\infty := I \cap G_\infty$, one have a canonical
$I_\infty$-equivariant $\F_p$-linear isomorphism
$$T_{\gs_1}(\M/p\M) \simeq \F_{p^d} \big(\theta_{d,0}^{n_0}
\theta_{d,1}^{n_1} \cdots \theta_{d,d-1}^{n_{d-1}} \big).$$
\end{enumerate}
\end{lemma}

\begin{proof}
In what follows, we will make an intensive use of results of
\cite{intro:kisin2}. Before beginning the proof, we would like to
emphasize that in the latter reference, $K_0$ is defined to be the
maximal absolutely unramified subextension of $K$, whereas in this
paper $K_0$ is just $K$. By chance, since we are assuming $e = 1$,
the two definitions coincide.

We define a $(\varphi, N_\nabla)$-module over $\O$ (see \emph{loc.
cit.}, \S 1.1 for the definition of all of this) by the following:
$$\left\{\begin{array}{l}
\calM = \O \mathcal E_0 \oplus \O \mathcal E_1 \oplus \cdots
\oplus \O \mathcal E_{d-1} \\
\varphi(\mathcal E_{i+1}) = (u+p)^{n_i} \mathcal E_i \quad (i \in
\Z/d\Z) \\
N_\nabla(\mathcal E_i) = \Bigg(\displaystyle \sum_{j=0}^\infty n_{i+j}
p^{j-1} u^{p^j} \lambda_j\Bigg) \mathcal E_i
\end{array}\right.$$
where $\lambda_j = \prod_{n \geq 0, n \neq j} \varphi^n(\frac u p + 1)$.
It is straitforward to check that the above $(\varphi, N_\nabla)$-module is well-defined (one only need to check that $ N_\nabla \varphi = (u+p)\varphi N_\nabla  $), and  that $\calM = \M \otimes_\gs \O$ as a $\varphi$-module. Now
let's compute the filtered $(\varphi,N)$-module $D(\calM)$ where $D$ is the
functor defined in \emph{loc. cit.}, \S 1.2.5. By definition $D(\calM) =
\calM / u \calM$ and it is equipped with Frobenius and monodromy
operator induced respectively by $\varphi$ and $N_\nabla$ on the
quotient. Computing the filtration is little more complicated. We first
need to determine the unique $\varphi$-compatible section $\xi : D(\calM)
\to \calM$ to the natural projection. Actually, one can easily check
that it is given by
$$\xi(\bar {\mathcal E}_i) = \Bigg(\prod_{j = 0}^{\infty} \varphi^j\Big(\frac u p + 1 \Big)^{n_{i+j}}\Bigg) {\mathcal E}_i \qquad (i \in \Z/d\Z),$$
where $\bar{\E}_i= (\E_i \mod u)  $ is the image of $\E_i$ in $D(\CM)$. Now applying recipe of \emph{loc. cit.}, \S 1.2.7, one see that
$\bar \E_{i+1} \in \Fil^{n_i} D(\calM)$. Therefore, if $D$ is equipped
with $N = 0$, the map $f : D \to D(\calM)$, $e_i \mapsto \bar{\mathcal
E_i}$ is compatible with $\varphi$, $N$ and filtration on both side.
Moreover, Theorem 1.3.8 of \emph{loc. cit.} shows that $D(\calM)$ is
admissible. Since the category of admissible filtered
$(\varphi,N)$-modules is abelian, $f$ has to be an isomorphism. Finally,
we use Proposition 2.1.5 of \emph{loc. cit.} to get (1).

The second part of the lemma is a simple computation left to the reader.
\end{proof}

\begin{proof}[Proof of Theorem \ref{theo:tame}]
We first assume $e=1$. Denote by $I$ the inertia subgroup of $G_K$. Let
$T$ be a tamely ramified representation of $G_K$ killed by $p$. Since
the tame inertia group is procyclic or order prime to $p$, $T|_I$ splits
as a direct sum of irreducible representations. By \cite{serre}, \S 1.7,
every irreducible representation of $I$ is isomorphic to
$$\F_{p^d} \big(\theta_{d,0}^{n_0} \theta_{d,1}^{n_1} \cdots
\theta_{d,d-1}^{n_{d-1}} \big)$$
for one sequence of integers between $0$ and $p-1$, periodic of period
$d$. Then applying Lemma \ref{lem:computkisin}, we construct a
$I_\infty$-equivariant isomorphism $f : L/pL \simeq T$ where $L$ is a
lattice in a crystalline representation with Hodge-Tate weights in $\{0,
\ldots, p-1\}$. We already know that the wild inertia subgroup $I_p$
does not act of $T$. On the other hand, using that $L/pL$ is a direct
sum of irreducible representations of $G_\infty$ and hence of $G_K$, a
standard argument shows that $I_p$ acts also trivially on $L/pL$. Thus
$f$ commutes with action of $I_p I_\infty = I$. Since $L/pL$ and $T$ are
finite dimensional over $\F_p$, they are finite and $f$ is
$G_{K'}$-equivariant for a finite unramified extension $K'$ of $K$. Let
$g$ be the composite morphism $L \to L/pL \simeq T$. Consider the map
$$\t{Ind}_{G_{K'}}^{G_K} L = \Z_p[G_K] \otimes_{\Z_p[G_{K'}]} L \to
T, \quad [\sigma] \otimes x \mapsto \sigma g(x).$$
It is apparently $G_K$-equivariant and surjective: it is a lift of $T$.
Furthermore, the restriction of $\t{Ind}_{G_{K'}}^{G_K} L$ to $G_{K'}$
is a direct sum of copies of $L$, and hence is crystalline with
Hodge-Tate weights in $\{0, \ldots, p-1\}$. Since $K'/K$ is unramified,
also is $\t{Ind}_{G_{K'}}^{G_K} L$ and we are done in this case.

Next, we assume $K/W[1/p]$ to be tamely ramified. Let $T$ be a
$p$-torsion representation of $G_K$ and set $T_0 =
\t{Ind}_{G_K}^{G_{W[1/p]}} T$. Using that $K/W[1/p]$ is tamely ramified,
we easily see that $T_0$ is still tamely ramified as a representation of
$G_{W[1/p]}$. Hence, one can apply the first part of the proof and then
find a lift $L \to T_0$ where $L$ is a lattice in a crystalline
representation of $G_{W[1/p]}$ with Hodge-Tate weights in $\{0, \ldots,
p-1\}$. Restricting actions to $G_K$, we get a surjective map $L \to
T^{[K:W[1/p]]}$ and then composing with the first projection we get a
lift of $T$, which is convenient since the restriction of a crystalline
representation is again crystalline with same Hodge-Tate weights.

Finally we go to the general case ($K$ arbitrary). Denote by $K'$ the
maximal tamely extension of $W[1/p]$ inside $K$. Let $K^\tr$ (resp.
${K'}^\tr$) be the maximal tamely ramified extension of $K$ (resp.
$K'$). Using that $K/K'$ is totally ramified of degree a power of $p$,
it is easy to check that $K \cap {K'}^\tr = K'$ and $K \: {K'}^\tr =
K^\tr$. We then deduce the existence of a canonical isomorphism between
$\gal(K^\tr/K)$ and $\gal((K')^\tr/K')$. Therefore, any tamely ramified
representation of $G_K$ has a natural prolongation to $G_{K'}$ which
remains tamely ramified. Theorem follows easily from this remark.
\end{proof}




\end{document}